\documentclass[11pt]{article}
\usepackage{amsmath,amsthm,amssymb,euscript,pifont}
\usepackage{mathrsfs}
\usepackage{xspace}
\usepackage{times}
\usepackage{authblk}
\usepackage{bm}
\newdimen\headwidth
\newdimen\headrulewidth
\headwidth \textwidth
\usepackage{color}
\usepackage[colorlinks=true]{hyperref}
\hypersetup{
    linkcolor = {blue},
    citecolor = {blue}
    }
\usepackage{fullpage}
\usepackage{xspace}

\usepackage{color}

\newcommand{\Bb}{\mathcal B}
\newcommand{\Cc}{\mathcal C}

\newcommand{\Ff}{\mathcal F}

\newcommand{\Pp}{\mathcal P}

\newcommand{\EE}{\mathbb E}
\newcommand{\NN}{\mathbb N}
\newcommand{\PP}{\mathbb P}
\newcommand{\RR}{\mathbb R}

\newcommand{\er} {\mathbb R}

\newcommand{\argmin}{\ensuremath{\operatorname{argmin}}}
\newcommand\1{\leavevmode\hbox{\rm \small1\kern-0.35em\normalsize1}}

\usepackage{dsfont}
\newcommand{\hypcoef}{({\it A}) }
\newcommand{\hypcoefi}{({\it $A_1$}) }
\newcommand{\hypcoefii}{({\it $A_2$}) }
\newcommand{\hyp}{({\it H}) }
\newcommand{\hypi}{({\it $H_1$}) }
\newcommand{\hypii}{({\it $H_2$}) }
\newcommand{\hypiii}{({\it $H_3$}) }

\theoremstyle{plain}
\newtheorem{theorem}{Theorem}[section]
\newtheorem{prop}[theorem]{Proposition}

\newtheorem{remark}[theorem]{Remark}
\newtheorem{lemma}[theorem]{Lemma}
\newtheorem{definition}[theorem]{Definition}
\newtheorem{corollary}[theorem]{Corollary}

\title{Diffusion processes with weak constraint through penalization approximation.}
\author{J.-F. Jabir\footnote{CIMFAV, Facultad de Ing., Universidad de Valpara\'iso, 222 General Cruz, Valpara\'iso, Chile; jean-francois.jabir@uv.cl. This author acknowledges the support of the FONDECYT INICIACI\'ON Project \textordmasculine  11130705, and the support of N\'ucleo Milenio MESCD.}}

\date\today
\begin{document}
\maketitle

\paragraph{Abstract:} In this paper, we investigate the construction of a diffusion process whose time-marginal densities are constrained to belong to a given set at all time. The construction is obtained from a penalization approximation to the constraint set, acting on the Wasserstein distance $W_2$ to the constraint space. Under some technical assumptions on the constraint space and the initial distribution of the model, the penalization approximation yields to a stochastic differential equation analogous to the Skorohod problem of reflected diffusion.\\ \\
\textbf{Keywords:} Constrained diffusion processes; Penalization method; Wasserstein space.\\
\textbf{AMS 2010 subject classifications}: 39A50, 35E10, 60J60.
\section{Introduction}
\subsection{General framework}
 In this work, we are interested in the construction of a couple of time-continuous stochastic processes $(X_{t},L_{t};\,0\leq t\leq T)$, defined up to an arbitrary finite time $T$, satisfying the following stochastic differential equation:
\begin{subequations}
\begin{equation}
X_{t}=X_{0}+\int_{0}^{t}b(s,X_{s})\,ds+\int_{0}^{t}\sigma(s,X_{s})\,dW_{s}+L_{t},\,X_0\sim\mu_0,\label{eq:DiffusionMotion}
\end{equation}
and satisfying the constraint that
\begin{equation}
\mbox{Law}(X_{t})\mbox{ belongs to }K,\,\mbox{for all}\,t\in(0,T],\label{eq:WeakConstraint-marginal}
\end{equation}
\end{subequations}
for $K$ a given subset of the space $\Pp(\er^d)$ of all probability measures defined on $\er^d$.
 In \eqref{eq:DiffusionMotion}-\eqref{eq:WeakConstraint-marginal}, $(W_{t};\,0\leq t\leq T)$ is a standard $\er^{d}$--Brownian motion, and the coefficients $b$ and $\sigma$ will be assumed to be given smooth functions (see the assumptions \hypcoefi and \hypcoefii below
 for the precise setting).

The system \eqref{eq:DiffusionMotion}-\eqref{eq:WeakConstraint-marginal} aims to describe, in a general way, a diffusion process submitted to a weak constraint, namely a
diffusion process $(X_t;\,0\leq t\leq T)$ whose time-marginal distributions are restricted to remain in a given subset of $\Pp(\er^d)$ and where $(L_{t};\,0\leq t\leq T)$ models a control component ensuring that the constraint \eqref{eq:WeakConstraint-marginal} is fulfilled. The terminology "weak constraint" refers here to a constraint on the law of the stochastic process in comparison to pathwise or strong constraint where paths of the process are constrained to remain in a given subdomain of $\er^d$.
 The time-marginal constraint \eqref{eq:WeakConstraint-marginal} can be seen as a very particular case of a more general class of weak conditioning involving path-distribution constraints of the form:
\begin{equation}
\mbox{Law}(X_{t};\,0\leq t\leq T)\in \mathbf{K},\,\mbox{for }\mathbf{K}\mbox{ a given subset of }\Pp(\Cc([0,T];\er^d)).\label{eq:WeakConstraint-complete}
\end{equation}
In this situation, the construction of diffusion processes satisfying a constraint of the form  \eqref{eq:WeakConstraint-complete}
appeared in various theoretical and applied situations such as in stochastic mechanics (see Cattiaux and L\'eonard \cite{CatLeo-95}); diffusion processes with conditioned initial-terminal distribution (see e.g. Baudoin \cite{Baudoin-02}, Mikami and Thieullen \cite{MikThi-08}, Tan and Touzi \cite{TanTou-13}); for the modeling of crowd motion with congestion phenomenon (Maury \textit{et al.} \cite{MauRouSan-10}, Santambrogio \cite{Santambrogio-15} and reference therein); Pdf methods for the simulation of incompressible turbulent flows (see Bossy {\it et al.} \cite{jabir-13a}),  ... The two latest reference feature very singular weak constraints and \cite{jabir-13a} was the initial motivation
of the present work. Let us also point out that Briand {\it et al.} \cite{BrChGuLa-16} addressed the well-posedness problem and particle approximation
of diffusion with mean reflection corresponding to the constraint ($K=\{\nu\in \,|\,\int h(x)\nu(dx)\geq 0\}$ for $h:\er^d\rightarrow \er$ a smooth function).
On a more general setting, C. L\'eonard has investigated in a series of articles (\cite{Leonard-03}, \cite{Leonard-08}, \cite{Leonard-10})
the general problem of minimization of entropy functional under linear constraints showing qualifications constraints for its dual formulations and its solvability.

Hereafter, our approach will be focused on the construction of a solution to \eqref{eq:DiffusionMotion}-\eqref{eq:WeakConstraint-marginal} through the introduction of an $\epsilon$-penalization approximation of the constraint \eqref{eq:WeakConstraint-marginal}, defined as follows: Given $\epsilon>0$, we consider $(X^\epsilon_t,L^\epsilon_t;\,t\in[0,T])$ satisfying
\begin{subequations}
\label{eq:PenalizedSDE}
\begin{align}
&X^\epsilon_t=X_0+\int_0^t b(s,X^\epsilon_s)\,ds+\int_0^t \sigma(s,X^\epsilon_s)\,dW_s+L^\epsilon_t,\label{eq:PenalizedSDEa}\\
&L^\epsilon_t=\int_0^t\frac{T^{\mu_\epsilon(s)\rightarrow \mu^K_\epsilon(s)}(X^\epsilon_s)-X^\epsilon_s}{\epsilon}\,ds,\label{eq:PenalizedSDEb}
\end{align}
\end{subequations}
where, for all $0\leq t\leq T$, \\
$\circ$ $\mu_\epsilon(t)$ is the distribution of $X^\epsilon_t$,\\
\noindent
$\circ$ $\mu^K_\epsilon(t)$ is the $W_2$-projection of $\mu_\epsilon(t)$ to the constraint space $K$ given by
\[
\mu^K_\epsilon(t)=\argmin_{\nu\in K}W_2(\mu_\epsilon(t),\nu),
\]
where $W_2$ is the Wasserstein distance with quadratic cost which, given two probability measures $\nu_0,\nu_1$, is defined as
\begin{equation}\label{eq:W2-bis}
W_2(\mu,\nu)=\sqrt{\inf_{X^\mu\sim\mu,\,X^\nu\sim \nu}\EE\left[|X_0-X_1|^2\right]},
\end{equation}
where the infimum is taken over all couples of random variables $(X^\mu,X^\nu)$ such that $\mbox{Law}(X^\mu)=\mu$ and $\mbox{Law}(X^\nu)=\nu$;

\noindent
$\circ$ $T^{\mu_\epsilon(t)\rightarrow \mu^K_\epsilon(t)}$ is the $W_2$-optimal transport mapping $\mu_\epsilon(t)$ towards $\mu^K_\epsilon(t)$; namely $T^{\mu_\epsilon(t)\rightarrow \mu^K_\epsilon(t)}$ is a Borel measurable $\er^d$-vector field such that
$\mu^K_\epsilon(t)=T^{\mu_\epsilon(t)\rightarrow\mu^K_\epsilon(t)}\#\mu_\epsilon(t)$ (of equivalently, the law of $T^{\mu_\epsilon(t)\rightarrow\mu^K_\epsilon(t)}$ under
$\mu^K_\epsilon(t)$ is equal to $\mu^K_\epsilon(t))$) and such that
\[
W_2(\mu_\epsilon(t),\mu^K_\epsilon(t))=\int \left|T^{\mu_\epsilon(t)\rightarrow \mu^K_\epsilon(t)}(x)-x\right|^2\mu_\epsilon(t,dx).
\]
The penalization system \eqref{eq:PenalizedSDE} features an formulation analog to classical penalized approximation for reflected diffusions or more generally for multivalued ODEs or SDEs (see Brezis \cite{Brezis-73}, Bernardin \cite{Bernardin-04}, C\'epa \cite{Cepa-95}, \cite{Cepa-98}, Slom\'inski \cite{Slominski-13}) meanwhile its construction relies strongly on the distance $W_2$ and its relationship with the theory of optimal mass transportation. Additionally,
the choice of $W_2$ to measure the penalization to a constraint set $K$ is here justified by the particular topological and convex properties and the related sub-differential and differential calculus defined on the space $(\Pp_2,W_2)$ developed in  Villani \cite{Villani-03}, \cite{Villani-09}, Ambrosio, Gigli and Savar\'e \cite{AmbGigSav-05}, Santambrogio  \cite{Santambrogio-15}. These notions will be discussed in Section \ref{sec:Preliminaries}.

\textbf{Notation:} \\
$\circ$ $\Pp(\er^d)$ will denote the set of probability measure defined on $(\er^d,\Bb(\er^d))$, and $\Pp^{ac}(\er^d)$ the subset of all probability measures absolutely continuous with respect to the Lebesgue measure. For any $\mu$ in $\Pp^{ac}(\er^d)$, $\frac{d\mu}{dx}$ is the probability density function of $\mu$ with respect to Lebesgue measure.

$\circ$ $\Pp_2$ [respectively $\Pp^{ac}_2$] will denote the subset of $\mu\in \Pp(\er^d)$ [$\mu\in \Pp^{ac}(\er^d)$] with finite second moments ($\int |x|^2\mu(dx)<\infty$).

$\circ$ For any probability measure $\mu$ defined on $\er^d$ and any $\er^d$-vector field $T$, $T\#\mu$ will denote the push-forward of $\mu$ along $T$ that is
\[
\int f(x) T\#\mu(dx) =\int f(T(x))\,\mu(dx),\,f\in \Cc_b(\er^d).
\]
$\circ$ We will denote by $L^2(\mu)$ the space of $\er^d$-vector fields such that $\int |T(x)|^2\mu(dx)<\infty$.

$\circ$ For any probability measure $\mu$ in $\Pp_2(\er^d)$, we define the $W_2$-distance of $\mu$ to $K$ as
\begin{equation}\label{def:DistanceK}
W_2(\mu,K)=\inf_{\nu\in K}W_2(\mu,\nu)
\end{equation}
and the projection $\mu^K$ of $\mu$ on $K$ as the minimizer of
\begin{equation}\label{def:ProjK}
\inf_{\nu\in K} W_2(\mu,\nu)
\end{equation}
whenever this minimizer is unique.

$\circ$ Given a probability measure $\mu$, we will often denote by $X^\mu$ an arbitrary random variable such that $X^\mu\sim\mu$.

$\circ$ $\Vert \Vert_{Lip}$ will denote the Lipschitz norm $\Vert f\Vert_{Lip}=\sup_{x,y\in\er^d,\,x\neq y}\left|f(x)-f(y)\right|/\left|x-y\right|$,
and $\Vert \Vert_{\infty}$ the supremum norm: $\Vert f\Vert_{\infty}=\sup_{x\in\er^d}|f(x)|$.

\subsection{Main results}
 The main results of this article concern the existence and uniqueness of a solution to \eqref{eq:PenalizedSDE} and the study of its behavior as the penalization order goes to $0$. Both problems will be considered under the following assumptions:

\noindent
\textbf{Assumptions on $b$ and $\sigma$}: The drift vector $b:(0,T)\times\er^{d}\rightarrow \er^{d}$ and the diffusion matrix $\sigma:(0,T)\times\er^{d}\rightarrow \er^{d}\times\er^{d}$ are Borel-measurable functions satisfy the following assumptions \hypcoef:
\begin{description}
\item[\hypcoefi] $b$ and $\sigma$ are bounded $\Cc^\infty$-functions with bounded first derivatives;
\item[\hypcoefii] $a=\sigma\sigma^{\star}$ is uniformly elliptic; that is there exists $\lambda>0$ such that
\begin{equation*}
\lambda|\eta|^{2}\leq \eta \cdot a(t,z) \eta,\,\forall \eta\in\er^{d}.
\end{equation*}
\end{description}
\noindent
\textbf{Assumptions on the constraint space $K$ and the initial distribution $\mu_0$}:
\begin{description}
\item[\hypi] $\mu_0$ is in $\Pp^{ac}_2(\er^d) \cap K$, the intersection being implicitly assumed to be non-empty.
\item[\hypii] $K$ is a closed subset of $\Pp(\er^{d})$ (equipped with the weak topology) and one of the two following conditions hold true: for all $\nu_0,\nu_1\in K$ and $X^{\nu_0}\sim \nu_0$, $X^{\nu_1}\sim \nu_1$, we have the family of measures
\[
\mbox{Law}\left((1-\alpha)X^{\nu_0}+\alpha X^{\nu_1}\right)
\]
belongs to $K$ for all $0\leq \alpha\leq 1$.
\item[\hypiii] The set
\begin{equation*}
\mbox{Int}(K):=\left\{
\begin{aligned}
&\mbox{The set of }\nu\in\Pp(\er^d)\,\mbox{such that there exists}\,r>0,\,Y\sim \nu\,\mbox{defined on some probability space} (\Omega,\Ff,\PP),\\
&\mbox{for which, for all r.v. }X:\Omega\rightarrow \er^d\mbox{ such that }\PP-\mbox{a.s.}\,|X|\leq 1,\mbox{ then }\,\mbox{Law}(Y+rX)\mbox{ is in } K.
\end{aligned}
\right\}
\end{equation*}
is not empty.
\end{description}

\begin{remark} Before stating the main results of this paper, let us comment the preceding assumptions \hypcoef and \hyp. \hypcoef ensures some simple preliminary properties on the distribution of the non-penalized process (that is \eqref{eq:PenalizedSDE} with $L^\epsilon \equiv 0$) and may be weakened for other type of hypo-elliptic generator.
The assumption \hypii feature topological and convex properties on the constraint set which ensure the existence and uniqueness of the $K$-projection $\mu^K$ of any measure $\mu$ in $\Pp(\er^d)$ (see Lemma \ref{lem:ProjK} below).

The assumption \hypiii can be understood as a non-empty interior assumption, usually assumed for reflected ODEs and SDEs (see e.g. \cite{Brezis-73}), in the sense that we are assuming that there exists at least one probability measure which can be freely transported in any direction $x+r f(x)$ within $K$ through a certain class of vector fields $T$.
 Examples of sets satisfying the assumptions \hypii, and \hypiii will be shown at the end of Section \ref{sec:Preliminaries}.
\end{remark}

On the wellposedness problem of \eqref{eq:PenalizedSDE} we then have the following result:

\begin{theorem}\label{thm:Main1} Assume that \hypcoef, \hypi, \hypii hold true. Then there exists a unique solution, in the pathwise sense, to \eqref{eq:PenalizedSDE}
\end{theorem}
Let us point out that \eqref{eq:PenalizedSDE} provide an original stochastic differential equation as the drift component depends implicitly on the time-marginal distribution of $(X^\epsilon_t;\,0\leq t\leq T)$. Since the optimal transport $T^{\mu_\epsilon(t)\rightarrow \mu^K_\epsilon(t)}$ is simply the identity function $x\in\er^d\mapsto x$ whenever $\mu_\epsilon(t)$ belongs to $K$, the component $(L^\epsilon_t;\,0\leq t\leq T)$ acts only when the distribution of $X^\epsilon_t$ is outside the constraint domain $K$ and the direction of $dL^\epsilon_t/dt$ points out $\mu_\epsilon(t)$ towards $K$ with minimal energy cost conditioning the distribution to lie in $K$ as $\epsilon $(see Section \ref{sec:Limit}). This reformulation will be more discussed in Section \ref{sec:Preliminaries}, as well as some monotone properties related to $\partial W^2_2(\mu,K)$ (see Section \ref{sec:Preliminaries}).
The wellposedness problem related to \eqref{eq:PenalizedSDE} will be handled by formally recasting the SDE into a multivalued SDE of the form
\begin{equation*}
\left\{
\begin{aligned}
&-dX^\epsilon_t+b(t,X^\epsilon_t)\,dt+\sigma(t,X^\epsilon_t)\,dW_t\in \frac{1}{2\epsilon}\bm{\partial} W^2_2(\mu^\epsilon(t),K),\\
&X_0\sim \mu_0,
\end{aligned}
\right.
\end{equation*}
where $\bm{\partial} W^2_2(\mu,K)$ is the sub-differential of $\mu\mapsto W^2_2(\mu,K)$ in the sense of \cite{AmbGigSav-05} (see Definition \ref{def:Subdifferentials} below),
and by exhibiting some monotone properties related to $\partial W^2_2(\mu,K)$ (see Corollary \ref{coro:Monotone1} in Section \ref{sec:Preliminaries}). Furthermore
$\mu\mapsto W^2_2(\mu,K)/2\epsilon$ can be seen as a Moreau-Yosida approximation of the (convex) indicator of $K$:
\begin{equation}
\label{eq:Indicator}
\delta_{K}(\mu)=\left\{
\begin{aligned}
&0\,\mbox{if}\,\mu\in K,\\
&+\infty\,\mbox{otherwise}.
\end{aligned}
\right.
\end{equation}
As the penalization order $\epsilon$ tend to $0$, it should be expected that the natural limit $(X_t,L_t;\,0\leq t\leq T)$ to \eqref{eq:PenalizedSDE} is given by the form
\begin{equation}
\label{eq:MultivaluedSDE}
\left\{
\begin{aligned}
&-dX_t+b(t,X_{t})\,dt+\sigma(t,X_{t})\,dW_{t}\in \bm{\partial} \delta_{K}(\mu(t))d|L|_t,\\
&\mu(t)=Law(X_{t}),
\end{aligned}
\right.
\end{equation}
where $(|L|_t;\,0\leq t\leq T)$ corresponds to the total variation of $(L_t;\,0\leq t\leq T)$ and $\bm{\partial}\delta_{K}(\mu)$ consists in all vector fields
$\xi:\er^d\rightarrow\er^d$ such that
\begin{equation*}
0\geq \int \xi(x)\cdot (T(x)-x)\,\mu(dx),\,\mbox{for all}\,T:\er^d\rightarrow\er^d\,\mbox{such that}\,T\#\mu\,\mbox{is in}\,\mathbf{K},
\end{equation*}
The general formulation of \eqref{eq:MultivaluedSDE} features an analog form of the stochastic differential equation with strong (path) constrained
(\cite{Brezis-73}, \cite{LioSzn-84}, \cite{Cepa-98}, \cite{Bernardin-04}), here adapted to
the particular framework of the weak constraint.

More rigorously, we show that

\begin{theorem}\label{thm:Main2} Assume that \hypcoef, \hypi, \hypii and \hypiii hold true. Then, as $\epsilon$ tends to $0$,
$(X^\epsilon_t,L^\epsilon_t;\,0\leq t\leq T)$ converges weakly toward a couple of processes $(X_t,L_t;\,0\leq t\leq T)$, defined on some filtered
probability space $(\Omega,\Ff, (\Ff_t;\,0\leq t\leq T),\PP)$, such that

$(i)$ $\PP$-a.s. $t\mapsto X_t$ is continuous and
\[
X_{t}=X_0+\int_0^t b(s,X_s)\,ds+\int_0^t\sigma(s,X_s)\,dW_s+L_t,\,t\in[0,T],
\]
where $X_0\sim \mu_0$ and $(W_t;\,0\leq t\leq T)$ is a $\er^d$-Brownian motion;

$(ii)$ for all $t$, $\mu(t):=\text{Law}(X_t)$ belongs to $K$;

$(iii)$ $(L_t;\,0\leq t\leq T)$ is a continuous process with bounded variations such that, for all $\Ff_t$-adapted continuous process $(Y_t;\,0\leq t\leq T)$ such that, for all $0\leq t\leq T$, $\mbox{Law}(Y_t)$ is in $K$, we have
\begin{align*}
\EE_\PP\left[\int_0^T\left(Y_{s}-X_s\right)\cdot dL_r\right]\leq 0.
\end{align*}
\end{theorem}

The paper is organized as follows: The next section is dedicated to a short account of some general topological properties of the space $(\Pp_2,W_2)$ and the related notions of (sub-)differential calculus which will be used to construct \eqref{eq:MultivaluedSDE}. In Section \ref{sec:PenalizedWellposedness}, we prove that the wellposedness result stated in Theorem \ref{thm:Main1} under \hypcoef, \hypi and \hypii. Next, Section \ref{sec:Limit} is dedicated to the proof of the limit behavior as $\epsilon\rightarrow 0^+$ of \eqref{eq:PenalizedSDE}.

\noindent

%
%

\section{Some recalls on the space $(\Pp_2,W_2)$ and preliminaries for the study of \eqref{eq:PenalizedSDE}}\label{sec:Preliminaries}

Over the past twenty years, the metric $W_2$, its link with the theory of optimal transportation and the particular topological and geometrical properties it ensure on $\Pp(\er^d)$ have been the subject of intensive and fruitful investigations and applications in various fields such as fluid mechanics, differential geometry, functional inequalities, gradient  flows equations on the space of probability measures and variational principle for nonlinear pdes, ... (see e.g. \cite{AmbGigSav-05}, and references therein). A particular feature related to transport of measures was the convex properties of functionals defined  $(\Pp_2,W_2)$ and differential calculus along variations of transported measures were first exhibited in McCann \cite{McCann-97}, Jordan, Kinderlehr and  Otto \cite{JorKinOtt-98}, Otto \cite{Otto-01} and more intensively studied in Villani \cite{Villani-03}, \cite{Villani-09}, and Ambrosio, Gigli and Savar\'e \cite{AmbGigSav-05}. Recently, Santambrogio  \cite{Santambrogio-15} provides similar properties along variations of probability measures.

More recently, a similar calculus of variations on the space $(\Pp_2,W_2)$ was introduced in Cardaliaguet \cite{Cardaliaguet-13} and Carmona and Delarue \cite{CarDel-15}, following the ideas of P.-L. Lions \cite{Lions-06}, in the general framework of Mean Field games and controlled McKean-Vlasov systems.  was considered in \cite{Lions-06}, \cite{Cardaliaguet-13}, \cite{CarDel-15} which provides There, derivatives of a functional $F:\Pp_2(\er^d)\rightarrow\er$ is obtained from the
\textit{lifting} of $F$; namely considering, on a probability space $(\Omega,\Ff,\PP)$ where $\Omega$ is a Polish space, $\Ff$ its Borel $\sigma$-algebra and $\PP$ is a atomeless Borel probability measure on $(\Omega,\Ff)$, the functional $\tilde{F}:L^2(\Omega,\Ff,\PP)\rightarrow\er$ which assigns to any random variable $X:\Omega \rightarrow\er^d$ such that $\PP\circ X^{-1}=\mu$ the value $\tilde{F}(X):=F(\mu)$. As pointed out by [\cite{Cardaliaguet-13}, Section $6$] and \cite{CarDel-15}, the choice of $(\Omega,\Ff,\PP)$ and of the
representant $X^\nu$ of $\nu$ is arbitrary and the functional $X\in L^2(\Omega,\Ff,\PP)\mapsto \tilde{F}(X)$ only depends on the law of $X$. This lifting technique provides a
particular suitable framework for stochastic calculus and some probabilistic interpretation of the various notion of convexity on the space $(\Pp_2,W_2)$ which will be used below.

In this section, we recall the notions of convex functionals on the space $(\Pp_2,W_2)$ and further sub-differential calculus on the space $(\Pp_2,W_2)$,
developed in \cite{Villani-03}, \cite{AmbGigSav-05}, \cite{Villani-09}, \cite{Santambrogio-15},
as well as some basic analytical properties related to the metric $W_2$.
The two last subsections are dedicated to a preliminary study of the functionals $\mu\mapsto W^2_2(\mu,K)$ and some simple
examples of constraint spaces satisfying the assumptions \hypii and \hypiii.

\subsection{Convex functionals and sub-differential calculus on $(\Pp_2,W_2)$}

The notion of convex functional $F:\Pp(\er^d)\rightarrow \er$ has been considered in various forms: Hereafter, a functional $F:\Pp(\er^d)\rightarrow \er$ will be said to be
\textbf{classically convex} if, for all $\mu,\nu$ in $\Pp(\er^d)$,
\[
F\left((1-\alpha)\mu+\alpha\nu\right)\leq (1-\alpha)F\left(\mu\right)+\alpha F\left(\nu\right),\,\forall\,\alpha\in[0,1].
\]
and a functional $F:\Pp(\er^d)\rightarrow \er$ will be said to be \textbf{$\lambda$-displacement convexity} for $\lambda$ in $\er$ (also referred to as \textbf{$\lambda$-geodesical convexity} in \cite{AmbGigSav-05}): if, for all $\mu,\nu$ in $\Pp(\er^d)$ and $T:\er^d\rightarrow\er^d$ a Borel vector field in $L^2(\mu;\er^d)$, for any $\alpha$ in $[0,1]$,
\[
F\left(T_\alpha\#\mu\right)\leq (1-\alpha)F\left(\mu\right)+\alpha F\left(T\#\mu\right)+\frac{\lambda}{2}\alpha(1-\alpha)W^2_2(\mu,\nu),
\]
where $T_\alpha$ is the convex interpolation
\[
T_\alpha(x)=(1-\alpha)x+\alpha T(x).
\]
The notion of displacement convex functional was first considered in McCann \cite{McCann-97} in order to characterize equilibrium states for the distribution of gas particles, and latter applied to define and construct steepest-descent schemes and gradient flows equations on $(\Pp_2,W_2)$ (see e.g. \cite{AmbGigSav-05} and references therein).

Let us point out that the notion of displacement convex translates in a very intuitive way as a classical notion of convexity in terms of "lifting" as
a functional $F:\Pp_2\rightarrow\er$ is displacement convex i.f.f. its lifting $\tilde{F}:L^2(\Omega,\Ff,\PP)\rightarrow\er$ is convex in the sense that, given $\mu,\nu\in\Pp_2$,
\[
\tilde{F}((1-\alpha)X^\mu+\alpha X^\nu)\leq (1-\alpha)  \tilde{F}(X^\mu)+\alpha \tilde{F}(X^\nu)
\]
for all $X^\mu\sim\mu$ and $X^\nu\sim\nu$, provided that $\mu$ is in $\Pp^{ac}_2$.
The class of functionals satisfying the displacement convexity property includes:
 \begin{description}
 \item[Potential energy:] $F(\mu)=\int V(x)\mu(dx)$ for $V$ a strongly convex function;
 \item[Interaction energy:] $F(\mu)=\int W(x-y)\mu(dx)\mu(dy)$ for $W:\er^d\rightarrow\er$ a strongly convex symmetric function;
 \item[Internal energy:] $F(\mu)=\int e(\frac{d\mu}{dx}(x)) dx$ for $e:(0,\infty)\rightarrow \er$ a Borel measurable function such that $e(0)=0$ and such that $r\in (0,\infty)\mapsto r^d e(r^{-d})$ is a convex non-increasing function;
\end{description}
(We refer the interested reader to [\cite{Villani-03}, Theorem $5.15$, Chapter 5], [\cite{AmbGigSav-05}, Chapter $7$ and Section $9.1$] and [\cite{Santambrogio-15}, Section $7.2$] for proof details.) Among these three examples, external and internal energy functionals are also classically convex meanwhile interaction energy do not exhibit such behavior.

\paragraph{Sub-differential calculus on $(\Pp_2,W_2)$:}
As different notions of convexity were introduced, different notions of variations and subsequent sub-differential calculus can be considered.
\begin{definition}\label{def:Subdifferentials}
$\circ$ Given a functional $F:\Pp(\er^d)\rightarrow  \er$ and $\mu$ in $\mbox{dom}(F):=\{\nu\in\Pp_2(\er^d)\,|\,|F(\nu)|<\infty\}$, then the sub-differential
$\partial F(\mu)$ is the set of scalar functions $\phi:\er^d\rightarrow\er$ belonging to the dual space $\Pp'(\er^d)$ and such that
\begin{equation}
F(\nu)-F(\mu)\geq \int \phi(x)\left(\nu(dx)-\mu(dx)\right),\,\forall\,\nu\in\Pp(\er^d).\label{eq:VerticalSubDerivativesFormula}
\end{equation}
$\circ$ Given a functional $F:\Pp_2(\er^{d})\rightarrow \er$ and $\mu\in dom(F)$, we define by $\bm{\partial}F(\mu)$ as the subset of all $\er^d$-vector fields
 $\xi$ in $L^2(\mu)$ such that
\begin{equation}
F(T\#\mu)-F(\mu)\geq \int \xi(x)\cdot \left(T(x)-x\right) \mu(dx),\,\forall\,T:\er^d\rightarrow\er^d.\label{eq:HorizontalSubDerivativesFormula}
\end{equation}
\end{definition}
The first notion of sub-differential consider variations along probability measures and the second definition is based on variations along $L^2$-transportation of measures.
 was first introduced in the seminal work Jordan, Kinderlehr and Otto \cite{JorKinOtt-98} and later deeply studied in \cite{AmbGigSav-05}.
 Let us observe that the notions of sub-differential are connected in the following way: assuming then that $\phi\in\partial F(\mu)$ is a $\Cc^1$-convex function on $\er^d$, it follows that
\begin{align*}
F(T\#\mu)-F(\mu)\geq \int \left(\phi(T(x))-\phi(x)\right)\mu(dx)\geq \int \nabla\phi(x)\cdot (T(x)-x) \mu(dx).
\end{align*}
Therefore, in a very heuristic form, we have that $\nabla\partial F(\mu)\subset\bm{\partial}F(\mu)$.
\subsection{Recall on some fundamental properties of the distance $W_2$:} Given two probability measures $\mu$ and $\nu$ in $\Pp(\er^d)$, $W_2(\mu,\nu)$, defined as in \eqref{eq:W2-bis}, measures the transportation cost of transporting $\mu$ towards $\nu$ relatively to the quadratic distance function $|x-y|^2$. Equivalently $W_2(\mu,\nu)$ formulates as
\begin{equation}\label{eq:W2}
W_2(\mu,\nu)=\sqrt{\inf_{\pi\in\Pi(\mu,\nu)}\int_{\er^d\times\er^d}|x-y|^2\pi(dx,dy)}
\end{equation}
where $\Pi(\nu_0,\nu_1)$ is the set of all couplings between $\mu$ and $\nu$; namely the set of all probability measures $\pi$ defined on $\er^d\times\er^d$ such that
\[
\int_{A\times\er^d}\pi(dx,dy)=\mu(A),\,\,\int_{\er^d\times B}\pi(dx,dy)=\nu(B),\,A,B\in\Bb(\er^d).
\]
Whenever $\mu$ and $\nu$ are in $\Pp_2(\er^d)$, the existence of a minimizing coupling $\underline{\pi}$ to the optimal transportation problem \eqref{eq:W2}
is always ensured (see more general case in \cite{Villani-09}).

 Assuming additionally that $\mu$ is absolutely continuous with respect to the Lebesgue measure, then (see e.g. [\cite{Villani-03}, Knott - Smith's and Brenier's Theorems 2.12]) there exists
 a mapping $T^{\mu\rightarrow \nu}:\er^d\rightarrow \er^d$ transporting $\mu$ towards $\nu$ (namely $T^{\mu\rightarrow \nu}\#\mu=\nu$) such that the optimal coupling $\underline{\pi}$ has full support on
 the $\{(x,y)\in\er^d\times\er^d\,|\,y=T^{\mu\rightarrow \nu}(x)\}$ and
 \begin{equation}
W^2_2(\mu,\nu)=\int \left|\nabla T^{\mu\rightarrow\nu}(x)-x\right|^2\,\mu(dx).\label{eq:MongeW2}
\end{equation}
The function $T^{\mu\rightarrow\nu}:\er^d\rightarrow \er^d$, known as Brenier's map, and defines an optimal transportation map between $\mu$ and $\nu$ in the sense
\[
\int \left|\nabla T^{\mu\rightarrow\nu}(x)-x\right|^2\,\mu(dx)=\inf \left\{\int \left|T(x)-x\right|^2\,\mu(dx)\,|\,T:\er^d\rightarrow\er^d\,\mbox{s.t.}\,T\#\mu =\nu\right\}.\label{eq:MongePb}
\]
The optimal transport $T^{\mu\rightarrow\nu}$ is further characterized as the sub-differential of the convex function
 $T^{\mu\rightarrow\nu}(x)=x-\nabla \Phi^{\mu\rightarrow\nu}(x)$ for (Lebesgue) a.e. $x$ in $\er^d$ where $\Phi^{\mu\rightarrow\nu}$ is the so called Kantorovich potential related to the dual formulation
 of $W_2$:
\begin{equation}
\label{eq:KantorovichPotential}
\begin{aligned}
W_2^2\left(\mu,\nu\right)&=\sup\left\{\int \Phi(x)\mu(dx)+\int \Psi(y)\nu(dy)\,|\,\Phi,\Psi \in \Cc_b(\er^d),\,\Phi(x)+\Psi(y)\leq |x-y|^2,\,\forall\,x,y\in\er^d\right\}\\
&=\int \Phi^{\mu\rightarrow\nu}(x)\mu(dx)+\int \Psi^{\nu\rightarrow\mu}(y)\nu(dy).
\end{aligned}
\end{equation}
Coming back to the probabilistic form \eqref{eq:W2-bis}, this result yields that, for any representant $X^\mu$ of $\mu$ ($X^\mu\sim \mu$)
defined on some probability space $(\Omega,\Ff,\PP)$, the distance $W_2(\mu,\nu)$ is given by
\begin{equation}\label{eq:W2Lifting}
W_2^2\left(\mu,\nu\right)=\EE_\PP\left[\left|T^{\mu\rightarrow\nu}(X^\mu)-X^\mu\right|^2\right].
\end{equation}
The topology generated by $W_2$ on $\Pp_2(\er^d)$ is stronger than the weak topology in $\Pp(\er^d)$ in the sense, for $\{\nu_n\}_{n\in\NN}$
 be a sequence of probability measures in $\Pp_2(\er^d)$, $\lim_{n\rightarrow +\infty}W_2(\nu_n,\nu)=0$
 if and only if $\nu_n$ converges weakly toward $\nu$ and $\lim_{n}\int |x|^2\nu_n(dx)=\int |x|^2\nu(dx)$.

Let us now comment on the convex properties of the functional $W_2$. For $\gamma$ a fixed probability measure in $\Pp^{ac}_2(\er^d)$, the $W_2$-distance to $\gamma$ defined by
 \[
 F_\gamma:\mu\in\Pp_2(\er^d)\mapsto F_\gamma(\mu)=W^2_2(\mu,\gamma),
 \]
 possess different convex properties. $F_\gamma$ is (classically) strictly convex (see [\cite{Santambrogio-15}, Proposition $7.19$]) and convex
 along generalized geodesics with basis $\gamma$ in $\Pp^{ac}_2(\er^d)$ ([\cite{AmbGigSav-05}, Section $9.2$]), but only the opposite distance $-F_\gamma$ is
 $\lambda$-displacement convex ([\cite{AmbGigSav-05}, Theorem $7.3.2$]).

\begin{corollary}\label{coro:ConvexProperties}[Convex properties of $F_\mu:\nu\mapsto W^2_2(\mu,\nu)$] Given $\mu\in\Pp_2(\er^d)$, the functional $\nu\in\Pp_2(\er^d)\mapsto -F_\mu(\nu)$ is displacement convex. If $\mu$ is in $\Pp^{ac}_2(\er^d)$ then, for any $\nu_0,\nu_1$ in $\Pp_2(\er^d)$, denoting by
$T^{\mu\rightarrow \nu_0}$ the $W_2$-optimal transport between $\mu$ and $\nu_0$, and by $T^{\mu\rightarrow \nu_1}$ the $W_2$-optimal coupling between $\mu$ and $\nu_1$, we have
\begin{equation}\label{eq:ConvexityalongGenGeo}
F_\mu(T^{\mu;\nu_0\rightarrow \nu_1}_\alpha\#\mu)\leq (1-\alpha)F_\mu(\nu_0)+\alpha F_\mu(\nu_1)-\alpha(1-\alpha)W^2_2(\nu_0,\nu_1),
\end{equation}
for
\[
T^{\mu;\nu_0\rightarrow \nu_1}_\alpha(x)=(1-\alpha)T^{\mu\rightarrow \nu_0}(x)+\alpha T^{\mu\rightarrow \nu_1}(x).
\]
In addition, $F_\mu$ is classically convex in the sense that, given $\mu\in\Pp^{ac}_2(\er^d)$, $\nu_0,\nu_1$ in $\Pp_2(\er^d)$, for $\nu_\alpha=(1-\alpha)\nu_0+\alpha\nu_1$, $0\leq \alpha\leq 1$,
\[
F_\mu(\nu_\alpha)\leq (1-\alpha)F_\mu(\nu_0)+\alpha F_\mu(\nu_1),
\]
and equality occurs if and only if $\nu_0=\nu_1$.
\end{corollary}
The displacement convexity of the opposite distance $-F_\gamma$ was obtained in [\cite{AmbGigSav-05}, Theorem $7.3.2$] and the property that
 $F_\gamma$ is (classically) strictly convex was proved [\cite{Santambrogio-15}, Proposition $7.19$]). The second notion of convexity, also referred as convexity
 along generalized geodesics with basis $\gamma$ in $\Pp^{ac}_2(\er^d)$, was shown in [\cite{AmbGigSav-05}, Section $9.2$]) in order to exhibit the convex properties of $\nu\mapsto W^2_2(\mu,\nu)$ along a particular class of interpolation between probability measures. We provide below an alternative proof of \cite{AmbGigSav-05} which relies on a appropriate lifting of the functional $F_\mu:\nu\mapsto W^2_2(\mu,\nu)$.

\begin{proof}[Proof of Corollary \ref{coro:ConvexProperties}] Consider $\tilde{F}:L^2(\Omega,\Ff,\PP)\times L^2(\Omega,\Ff,\PP) \rightarrow\er$ given by
\[
\tilde{F}(X,Y)=\EE_\PP\left[\left|X-Y\right|^2\right]
\]
which is the lifting of the functional $F:\Pp_2(\er^d)\times\Pp_2(\er^d)\rightarrow\er$ defined by
\[
F(\mu,\nu)=\int|x-y|^2\pi(dx,dy),\,\mu=Law(X),\,\nu=Law(Y),
\]
for $\pi=\PP\circ(X,Y)^{-1}$. Given $\mu\in \Pp^{ac}_2$ and $\nu_0,\nu_1$ in $\Pp_2(\er^d)$, the hilbertian identity yields that
\begin{equation}\label{ProofStep1g}
 \begin{aligned}
 &\EE_\PP\left[\left|(1-\alpha)X^{\nu_0}+\alpha X^{\nu_1}-X^\mu\right|^2\right]\\
 &=(1-\alpha)\EE_\PP\left[\left|X^{\nu_0}-X^\mu\right|^2\right]+\alpha
 \EE_\PP\left[\left| X^{\nu_1}-X^\mu\right|^2\right]-\alpha(1-\alpha)\EE_\PP\left[\left|X^{\nu_0}-X^{\nu_1}\right|^2\right],
 \end{aligned}
 \end{equation}
 for any $X^\mu$, $X^{\nu_0}$, $X^{\nu_1}$ defined on $(\Omega,\Ff,\PP)$. Taking $X^{\nu_i}=T^{\mu\rightarrow \nu_i}(X^\mu),\,i=0,1$,
 immediately yields that
 \begin{align*}
 &W^2_2(T^{\mu;\nu_0\rightarrow \nu_1}_\alpha\#\mu,\mu)\\
 &\leq \EE_\PP\left[\left|(1-\alpha)X^{\nu_0}+\alpha X^{\nu_1}-X^\mu\right|^2\right]%
 =(1-\alpha)W^2_2(\mu,\nu_0)+\alpha W^2_2(\mu,\nu_1)-\alpha(1-\alpha)\EE_\PP\left[\left|X^{\nu_0}-X^{\nu_1}\right|^2\right]\\
 &\leq (1-\alpha)W^2_2(\mu,\nu_0)+\alpha W^2_2(\mu,\nu_1)-\alpha(1-\alpha)W^2_2(\nu_0,\nu_1).
 \end{align*}
 On the other hand, choosing, in \eqref{ProofStep1g}, $X^{\nu_1}=T^{\nu_0\rightarrow \nu_1}(X^{\nu_0})$ and $X^\mu$ so that the couple of random variables
 $((1-\alpha)X^{\nu_0}+\alpha T^{\nu_0\rightarrow \nu_1}(X^{\nu_0}),X^\mu)$ achieves the optimal coupling between $\nu_\alpha=((1-\alpha)x+\alpha T^{\nu_0\rightarrow \nu_1}(x))\#\nu_0$ and $\mu$, gives
 \begin{align*}
 &W^2_2(\nu_\alpha,\mu)=\EE_\PP\left[\left|(1-\alpha)X^{\nu_0}+\alpha X^{\nu_1}-X^\mu\right|^2\right]\\
 &=(1-\alpha)\EE_\PP\left[\left|X^{\nu_0}-X^\mu\right|^2\right]+\alpha
 \EE_\PP\left[\left| X^{\nu_1}-X^\mu\right|^2\right]-\alpha(1-\alpha)W^2_2(\nu_0,\nu_1)\\
 &\geq (1-\alpha)W^2_2(\mu,\nu_0)+\alpha W^2_2(\mu,\nu_1)-\alpha(1-\alpha)W^2_2(\nu_0,\nu_1).
 \end{align*}
 Finally, to prove the strict classical convexity of $F_\mu$, let $(X^{\mu},X^{\nu_0})$ be a couple of r.v. achieving the optimal coupling of $W_2(\mu,\nu_0)$ and $(\tilde{X}^{\mu},\tilde{X}^{\nu_1})$ be a couple of r.v. achieving the optimal coupling of $W_2(\mu,\nu_0)$. In addition, let $\beta$ be a Bernouilli r.v. with parameter $\alpha$, independent of $(X^{\mu},X^{\nu_0},\tilde{X}^{\mu},\tilde{X}^{\nu_1})$. Then, since $(1-\beta)X^{\nu_0}+\beta \tilde{X}^{\nu_1}$ is a representant of $\nu_\alpha$ and $(1-\beta)X^{\mu}+\beta \tilde{X}^{\mu}$ is a representant of $\mu$,
 \begin{align*}
 (1-\alpha)W^2_2(\mu,\nu_0)+\alpha W^2_2(\mu,\nu_1)&=(1-\alpha)\EE\left[|X^\mu-X^{\nu_0}|^2\right]+\alpha\EE\left[|X^\mu-X^{\nu_1}|^2\right]\\
 &=\EE\left[|(1-\beta)X^\nu_0+\beta \tilde{X}^{\nu_1}-(1-\beta)X^\mu-\beta \tilde{X}^\mu|^2\right]\\
 &\geq W^2_2(\mu,\nu_\alpha).
 \end{align*}
 The equality holds true if and only if the optimal coupling between $\mu$ and $\nu_\alpha$ is achieved with
 \[
 \left((1-\beta)X^\mu-\beta \tilde{X}^\mu,(1-\beta)X^\nu_0+\beta \tilde{X}^{\nu_1}\right),
 \]
 or equivalently, since $\mu$ is in $\Pp^{ac}_2$, with
 \[
 \left((1-\beta)X^\mu-\beta \tilde{X}^\mu,(1-\beta)T^{\mu\rightarrow \nu_0}(\tilde{X}^{\mu})+\beta T^{\mu\rightarrow \nu_1}(\tilde{X}^{\mu})\right).
 \]
 But, according to the representation formula \eqref{eq:W2Lifting}, this immediately implies that, $\PP$-a.s.,
 \[
 (1-\beta)T^{\mu\rightarrow \nu_0}(\tilde{X}^{\mu})+\beta T^{\mu\rightarrow \nu_1}(\tilde{X}^{\mu})=T^{\mu\rightarrow \nu_\alpha}((1-\beta)X^\mu+\beta \tilde{X}^{\mu})
 \]
 and that $T^{\mu\rightarrow \nu_0}$ is a multivalued function, which is in contradiction with the characterization \eqref{eq:MongeW2} of the optimal transportation between $\mu$ and $\nu_\alpha$.
 \end{proof}

Let us finally, the (approximate) sur-differentiability and differentiability of $\mu\mapsto W_2(\mu,\gamma)$ was obtained in [\cite{AmbGigSav-05}, Chapters $7$ and $10$]
and [\cite{Santambrogio-15}, Proposition $7.17$]) observed that the sub-differential $\partial F_\gamma(\mu)$ is reduced to the Kantorovich potential $\psi^{\mu\rightarrow \gamma}$ related to the transport of $\mu$ to $\gamma$; the proof of this result is obtained from the dual formulation of the $W_2$-optimal coupling problem and simply follows from the inequality:
\begin{equation}
\label{eq:SantambrogioSubDiff}
\begin{aligned}
&W^{2}_2(\nu,\gamma)-W^{2}_2(\mu,\gamma)\\
&=\sup\left\{\int \psi(x)\nu(dx)+\int \psi^*(x)\gamma(dx)\right\}-\sup\left\{\int \psi(x)\mu(dx)+\int \psi^*(x)\gamma(dx)\right\}\\
&\geq \int \psi^{\mu\rightarrow\gamma}(x)\left(\nu(dx)-\mu(dx)\right),
\end{aligned}
\end{equation}
for all $\nu$ in $\Pp(\er^d)$ with compact support.
\subsection{On the functional $\mu\mapsto W^2_2(\mu,K)$}
Let us begin this subsection, with some comments on the projection $\mu^K$ of $\mu$ on $K$ given by \eqref{def:ProjK}, under the assumption \hypii.
\begin{lemma}\label{lem:ProjK} Assume that $K$ satisfies either \hypii. Then, for all $\mu\in \Pp^{ac}_2(\er^d)$, there exists a unique probability measure $\mu^K$ in $\Pp_2(\er^d)$ such that
\[
W_2(\mu,\mu^K)=\inf\left\{W^2_2(\mu,\nu)\,|\,\nu\in K\right\}.
\]
In addition, for any r.v. $X^\mu$ defined on $(\Omega,\Ff,\PP)$ such that $X^\mu\sim\mu$, there exists a unique r.v. $X^{\mu^K}\sim\mu^K$ for which
\[
\EE_\PP\left[|X^\mu-X^{\mu^K}|^2\right]\leq \EE_\PP\left[|X^\mu-Y|^2\right],
\]
for all r.v. $Y:\Omega\rightarrow \er^d$ such that $\mbox{Law}(Y)$ is in $K$.
\end{lemma}
\begin{proof}
Given $\mu$ in $\Pp_2(\er^d)$, any minimizing sequence $\{\nu_n\}_{n\in\NN}$ of
\[
(P)\,\,\,\inf\left\{W^2_2(\mu,\nu)\,|\,\nu\in K\right\},
\]
has uniformly integrable first moment. Since $\nu\mapsto W_2(\nu,\mu)$ is lower semi-continuous for the weak convergence in $\Pp(\er^d)$,
the closedness of $K$ is sufficient to ensure the existence of at least one minimizer $\mu^K$ solving $(P)$.
Under \hypii, assume that there exist two different minimizers $\mu^K_1$ and $\mu^K_2$. Then as $\mu$ is in $\Pp^{ac}_2$, the optimal transport $T^{\mu\rightarrow \nu_i}$ mapping $\mu$ towards $\mu^K_i$ for $i=1,2$, exist. Applying \eqref{eq:ConvexityalongGenGeo} in Lemma \ref{coro:ConvexProperties}, we get
\begin{align*}
F_\mu(T^{\mu;\mu^K_1\rightarrow \mu^K_2}_\alpha\#\mu)&\leq (1-\alpha)F_\mu(\mu^K_1)+\alpha F_\mu(\mu^K_2)-\alpha(1-\alpha)W^2_2(\mu^K_1,\mu^K_2)
&<(1-\alpha)F_\mu(\mu^K_1)+\alpha F_\mu(\mu^K_2)
\end{align*}
as $\mu^K_1$ and $\mu^K_2$ are distincts.
Due to \hypii-(b), $\mbox{Law}\left((1-\alpha)T^{\mu\rightarrow \mu^K_1}(X^\mu)+\alpha T^{\mu\rightarrow \mu^K_2}(X^\mu)\right)=T^{\mu;\mu^K_1\rightarrow \mu^K_2}_\alpha\#\mu$ is in $K$
and the preceding strict inequality contradict the fact that $(P)$ can be achieved by two different measures $\mu^K_1$ and $\mu^K_2$.

Since $\mu$ is absolutely continuous with respect to the Lebesgue measure, choosing $X^{\mu^K}=T^{\mu\rightarrow \mu^K}(X^\mu)$ gives
$W_2$-optimal transport mapping $\mu$ to $\mu^K$ exists and we have, for any r.v. $X$ such that $\mbox{Law}(X)=\mu$,
\begin{align*}
\EE_\PP\left[\left|X^\mu-X^{\mu^K}\right|^2\right]&=W^2_2(\mu,K)\\
&=\inf\left\{\EE_\PP\left[|X^\mu-Y|^2\right]\,\mbox{such that}\,Y\,\mbox{is a }\er^d\mbox{-valued r.v. defined on }\Omega\mbox{ with}\,\mbox{Law}(Y)\in K\right\}.
\end{align*}
\end{proof}
Additionally, we have
\begin{lemma}\label{lem:Sub1} Define the functional $F_K:\Pp_2(\er^d)\rightarrow\er$ as
\begin{equation}
\label{eq:DistanceFunction}
F_K(\mu)=W^2_2(\mu,K),
\end{equation}
where $W_2(\mu,K)$ is given as in \eqref{def:DistanceK}.  Under \hypii, for any probability measure $\mu$ in $\Pp_2(\er^d)$, we have
\[
F_K(\nu)-F_K(\mu)\geq \EE_\PP\left[(X^\mu-X^{\mu^K})\cdot (X^\mu-X^\nu)\right],\,\forall\,\nu\in\Pp(\er^d).
\]
\end{lemma}
from which we also deduce that
\begin{corollary}\label{coro:Monotone1} Under \hypii, for $\mu,\nu$ in $\Pp^{ac}_2$, let $X^\mu$ and $X^{\nu}$ be arbitrary representants
of $\mu$ and $\nu$ respectively, and $X^{\mu^K}$ and $X^{\nu^K}$, their projections on $K$ given as in Lemma \ref{lem:ProjK}. Then we have
\begin{equation}\label{eq:prop:Monotone1}
0\geq \EE_\PP\left[\left(X^\nu-X^\mu\right)\cdot \left(\left(X^{\nu^K}-X^\nu\right)-\left(X^{\mu^K}-X^\mu\right)\right)\right].
\end{equation}
\end{corollary}
\begin{proof}[Proof of Lemma \ref{lem:Sub1}] Taking $X^{\nu^K}\sim \nu^K$, owing to \hypii-(b), $\mbox{Law}((1-\alpha)X^{\mu^K}+\alpha X^{\nu^K})$ is in $K$. Therefore, for
$\left((1-\alpha)x+\alpha T^{\mu\rightarrow \nu}(x)\right)\#\mu$, we have
\begin{align*}
&\EE_\PP\left[\left|(1-\alpha)\left(X^{\mu}-X^{\mu^K}\right)+\alpha\left(X^{\nu}-X^{\nu^K}\right)\right|^2\right]\\
&= (1-\alpha)^2\EE_\PP\left[\left|X^{\mu}-X^{\mu^K}\right|^2\right]
+\alpha^2\EE_\PP\left[\left|X^{\nu}-X^{\nu^K}\right|^2\right]+2\alpha(1-\alpha)
\EE_\PP\left[\left(X^{\mu}-X^{\mu^K}\right)\cdot \left(X^{\nu}-X^{\nu^K}\right)\right]
\\
&\leq (1-\alpha)\EE_\PP\left[\left|X^{\mu}-X^{\mu^K}\right|^2\right]
+\alpha\EE_\PP\left[\left|X^{\nu}-X^{\nu^K}\right|^2\right]\leq (1-\alpha)W^2_2(\mu,K)+\alpha W^2_2(\nu,K),
\end{align*}
applying the inequality $2a\cdot b\leq |a|^2+|b|^2,\,a,b\in\er^d$ for the third inequality.
Therefore
\begin{equation*}
F_K(\nu)-F_K(\mu)\geq \lim_{\alpha\rightarrow 0}\frac{1}{\alpha}\left(\EE_\PP\left[\left|(1-\alpha)\left(X^{\mu}-X^{\mu^K}\right)+\alpha\left(X^{\nu}-X^{\nu^K}\right)\right|^2\right]- W^2_2(\mu,K)\right).
\end{equation*}
Since
\[
W^2_2(\mu,K)\leq \EE_\PP\left[\left|X^{\mu}-(1-\alpha)X^{\mu^K}-\alpha X^{\nu^K}\right|^2\right]
\]
we get
\begin{align*}
F_K(\nu)-F_K(\mu)&\geq\lim_{\alpha\rightarrow 0}\frac{1}{\alpha}\left(\EE_\PP\left[\left|(1-\alpha)\left(X^{\mu}-X^{\mu^K}\right)+\alpha\left(X^{\nu}-X^{\nu^K}\right)\right|^2\right]- W^2_2(\mu,K)\right)\\
&\geq \lim_{\alpha\rightarrow 0}\frac{1}{\alpha}\left(\EE_\PP\left[\left|(1-\alpha)\left(X^{\mu}-X^{\mu^K}\right)+\alpha\left(X^{\nu}-X^{\nu^K}\right)\right|^2\right]-
\EE_\PP\left[\left|X^{\mu}-(1-\alpha)X^{\mu^K}-\alpha X^{\nu^K}\right|^2\right]\right).
\end{align*}
Adding and subtracting $\EE_\PP\left[\left|X^\mu-X^\nu\right|^2\right]$ to the right-hand expression gives
\begin{align*}
&F_K(\nu)-F_K(\mu)\\
&\geq \lim_{\alpha\rightarrow 0}\frac{1}{\alpha}\left(\EE_\PP\left[\left|(1-\alpha)\left(X^{\mu}-X^{\mu^K}\right)+\alpha\left(X^{\nu}-X^{\nu^K}\right)\right|^2\right]-
\EE_\PP\left[\left|X^{\mu}-X^{\mu^K}\right|^2\right]\right)\\
&\quad -\lim_{\alpha\rightarrow 0}\frac{1}{\alpha}\left(\EE_\PP\left[\left|X^{\mu}-(1-\alpha)X^{\mu^K}-\alpha X^{\nu^K}\right|^2\right]-\EE_\PP\left[\left|X^{\mu}-X^{\mu^K}\right|^2\right]\right)\\
&\geq 2\EE_\PP\left[\left(\left(X^{\nu}-X^{\nu^K}\right)-\left(X^{\mu}-X^{\mu^K}\right)\right)\cdot\left(X^{\mu}-X^{\mu^K}\right)\right]-
2\EE_\PP\left[\left(X^{\mu^K}-X^{\nu^K}\right)\cdot \left(X^{\mu}-X^{\mu^K}\right)\right]\\
&\geq 2\EE_\PP\left[\left(X^{\nu}-X^{\mu}\right)\cdot\left(X^{\mu}-X^{\mu^K}\right)\right].
\end{align*}
\end{proof}

Finally, we end this part with the following proposition:
\begin{prop}\label{prop:Monotone2} Under \hypii, for all $\mu$ in $\Pp_2$, $X^\mu$ and $X^{\mu^K}$ as in Lemma \ref{lem:ProjK},
\begin{equation}\label{eq:prop:Monotone3}
0\geq -\EE_\PP\left[\left(X^\nu-X^\mu\right)\cdot \left(X^{\mu^K}-X^\mu\right)\right]
\end{equation}
for all r.v. $X^\nu:\Omega\rightarrow \er^d$ with $\mbox{Law}(X^\nu)\in K$.
\end{prop}
\begin{proof}[Proof of Proposition \ref{prop:Monotone2}]
Under \hypii, $\mbox{Law}((1-\alpha)X^K+\alpha X^\nu)$ belongs to $K$ for all $0\leq \alpha\leq 1$. Hence, according to Lemma \ref{lem:ProjK},
\begin{equation*}
\EE_\PP\left[|X^\mu-X^{\mu^K}|^2\right]\leq \EE_\PP\left[|X^\mu-X^{\mu^K}+\alpha(X^\nu-X^{\mu^K})|^2\right]
\end{equation*}
  Dividing both sides of the inequality by $\alpha$ and taking the limit $\alpha\rightarrow 0$ gives
\begin{align*}
0&\leq 2\EE_\PP\left[\left(X^\mu-X^{\mu^K}\right)\cdot \left(X^\nu-X^{\mu^K}\right)\right]\\
&\leq 2\EE_\PP\left[\left(X^\mu-X^\nu\right)\cdot \left(X^\nu-X^{\mu^K}\right)\right]+\EE_\PP\left[\left|X^\nu-X^{\mu^K}\right|^2\right]
\end{align*}
from which we deduce \eqref{eq:prop:Monotone3}.
\end{proof}

\subsection{Examples of constraint space satisfying \hypii and \hypiii}
In this subsection, we present some examples of constraint sets satisfying assumptions \hypii and \hypiii.
\paragraph{Constraints on the support of a probability measure:} In this paragraph, we consider the case of constraint on the support of probability and recover classical results (see e.g. \cite{LioSzn-84}, \cite{Slominski-13}) on penalization approximation for the Skorokhod problem. Given $\mathbf{C}$ a compact convex subset of $\er^d$, the constraint set
\[
K=\left\{\nu\in\Pp(\er^d)\,|\,\mbox{the support of }\nu\,\mbox{is in }\mathbf{C}\right\}.
\]
satisfies naturally the assumptions \hypii and \hypiii: Indeed, the support of probability measures is closed for the weak topology, any couple $\mu,\nu$ in $K$ has convex interpolation in $K$ and for any random variable $X,Y$ in $\mathbf{C}$, then, by convexity of $\mathbf{C}$, $(1-\alpha)X+\alpha Y$ is in $\mathbf{C}$. For \hypiii, taking $\mathbf{\gamma}$ a probability measure such that the distance between the support of $\mathbf{\gamma}$ and the interior of $\mathbf{C}$ is strictly less than $\kappa>0$, the set of admissible transports for $\delta_K$ are of the form $x+\kappa F(x)$ for $F\in L^\infty(\er^d;\er^d)$ such that $\Vert F\Vert_{L^\infty}\leq 1$.
\paragraph{Constrained potential energy:} Given $0<\kappa_1<\kappa_2<\infty$ and a convex function $V:\er^d\rightarrow [0,\infty)$ of class $\Cc^1$
such that there exists a point $a\in\er^d$ in which $\kappa_1<V(a)<\kappa_2$ and such that $|\nabla V(x)|\leq c(1+|x|^p)$ for $1\leq p<\infty$, the set
\[
K=\left\{\nu\in\Pp(\er^d)\,|\, \int V(z)\nu(dz)\leq \kappa_2\right\}
\]
satisfies the assumption \hypii. The assumption \hypiii is satisfied for $\gamma(dx)=\delta_{\{a\}}(dx)$, the Dirac measure in $a$, and $r$ small enough

\paragraph{Constrained potential energy:} Given $\kappa>0$, and a symmetric convex function $W:\er^d \rightarrow [0,\infty)$ with bounded derivative such that $W(0)=0$, consider the set
\[
K=\left\{\nu\in\Pp(\er^d)\,|\,\iint W(x-y)\nu(dx)\nu(dy)\leq \kappa\right\}.
\]
Since $\nu\mapsto \iint W(x-y)\nu(dx)\nu(dy)$ is displacement convex, $K$ satisfies \hypii-(b). Taking
$\gamma(dx)=\delta_{\{x_0\}}(dx)$, for $x_0\in\er^d$, and observing that
\[
\iint W(x-y)\gamma(dx)\gamma(dy)=W(a-a)=0,
\]
and that for $X^\gamma$, $\tilde{X}^\gamma$ two independent representants of $\gamma$ and $Z,\tilde{Z}:\Omega\rightarrow \er^d$ two independent r.v., bounded a.s. by $1$,
\[
\EE_\PP\left[W(X^\gamma+rZ-\tilde{X}^\gamma-r\tilde{Z})\right]=
\EE_\PP\left[ W(r(Z-\tilde{Z}))\right]\leq r\Vert \nabla W\Vert_\infty,
\]
the assumption \hypiii is fulfilled for $r$ small enough.
\section{Existence and uniqueness result for the penalized system \eqref{eq:PenalizedSDE}.}\label{sec:PenalizedWellposedness}
In this section, we demonstrate the pathwise wellposedness of \eqref{eq:PenalizedSDE} as stated Theorem \ref{thm:Main1}, assuming that \hypcoef, \hypi and \hypii hold true. The proof will be decomposed into the existence part and the uniqueness part. The pathwise wellposedness of \eqref{eq:PenalizedSDE} will then follow from Yamada-Watanabe's results (see e.g. \cite{KarShr-98}).

The existence part is mostly handled by a time-approximation of \eqref{eq:PenalizedSDE} on a partition
\[
0=t^{N}_0\leq t^N_1\leq \cdots t^N_{N-1}\leq t^N_N=T,\,\sup_{0\leq i\leq N}|t^N_i-t^N_{i-1}|\leq 1/N,
\]
where the component $(L^\epsilon_t;\,0\leq t\leq T)$ is "frozen" in each interval $[t^N_i,t^N_{i+1})$. Due to assumptions \hypcoef, this approximation is well-posed and, using formulations in terms of a martingale problem, it can checked its convergence towards a weak solution to \eqref{eq:PenalizedSDE}.
The uniqueness part relies on the monotone property related to $\mu\mapsto 2W^2_2(\mu,K)/\epsilon$.

\subsection{Existence part} Given $N$ a positive integer, let us define the time-step $h:=T/N$ and the decomposition $\{t_{n}:=nh\}_{0\leq n\leq N}$ of $[0,T]$. Let us also construct, on some probability space $(\Omega,\Ff,\PP)$ under which are defined $X^\epsilon_0\sim\mu_0$ and an independent standard $\er^d$-Brownian motion $(W_t;\,t\geq 0)$, the process $(X^{N,\epsilon}_t;\,t\in[0,T])$ as follows:

$\circ$ for $0< t\leq t_1$
\[
X^{N,\epsilon}_t=X^\epsilon_0+\int_0^{t}b(s,X^{N,\epsilon}_s) \,ds+\int_{0}^t \sigma(s,X^{N,\epsilon}_s)\,dW_s,
\]
$\circ$ for $t_n< t\leq t_{n+1}$, given $X^{N,\epsilon}_{t_n}$ and its distribution $\mu^{N,\epsilon}(t_n)$,
\[
X^{N,\epsilon}_t=X^{N,\epsilon}_{t_{n}}+\int_{t_{n}}^t b(s,X^{N,\epsilon}_s) \,ds+\int_{t_{n}}^t \sigma(s,X^{N,\epsilon}_s)\,dW_s+L^{N,\epsilon}_t,
\]
for
\[
L^{N,\epsilon}_t=\int_{t_{n}}^t \frac{T^{\mu^{\epsilon}(t_n)\rightarrow \mu^K_\epsilon(t_n)}(X^{N,\epsilon}_{t_n})-X^{N,\epsilon}_{t_n}}{\epsilon}\,ds.
\]
The assumptions \hypcoefi and \hypcoefii and [\cite{Friedman-06}, Theorem 5.4, Chapter $5$] ensure that, for all $n\geq 1$, the distribution of $X_{t_n}$ is in $\Pp^{ac}_2(\er^d)$ so that the optimal transport $T^{\mu^{\epsilon}(t_n)\rightarrow \mu^K_\epsilon(t_n)}$ is  well defined. Additionally, \hypi ensures that we have
\begin{equation}\label{ProofStep1a}
\begin{aligned}
&\EE_\PP\left[\int_0^T\left|\frac{T^{\mu^{\epsilon}(t_n)\rightarrow \mu^K_\epsilon(t_n)}(X^{N,\epsilon}_{t_n})-X^{N,\epsilon}_{t_n}}{\epsilon}\right|^2\,ds\right]\\
&\leq \frac{T}{\epsilon^2 N}\sum_{n=0}^NW^2_2(\mu^{N,\epsilon}(t_n),K)\leq \frac{T}{\epsilon^2}\sup_{t\in[0,T]}W^2_2(\mu^{\epsilon}(t),K)<\infty.
\end{aligned}
\end{equation}
 Furthermore, defining
\[
\eta_N:t\in[0,T]\mapsto \eta_N(t):=h\lfloor t/h\rfloor,
\]
the assumptions \hypi and \hypcoefi yield that, for all $0\leq T_0\leq T$
\begin{align*}
&\EE_\PP\left[\max_{t\in[0,T_0]}|X^{N,\epsilon}_t|^2\right]=\EE_\PP\left[\max_{t\in[0,T_0]}\left|X^\epsilon_0+L^{N,\epsilon}_{\eta_N(s)}+\int_0^tb(s,X^{N,\epsilon}_s)\,ds+\int_0^t\sigma(s,X^{N,\epsilon}_s)\,dW_s\right|^2\right]\\
&\leq 8\left(\EE_\PP\left[|X^\epsilon_0|^2\right]+\int_0^{T_0}\EE_\PP\left[\left|\frac{T^{\mu^{N,\epsilon}(\eta_N(t))\rightarrow P_K(\mu^{N,\epsilon}(\eta_N(t)))}(X^{N,\epsilon}_{\eta_N(t)})-X^{N,\epsilon}_{\eta_N(t)}}{\epsilon}\right|^2\right]\,dt\right)\\
&\quad +8\EE_\PP\left[\max_{t\in[0,T_0]}|\int_0^tb(s,X^{N,\epsilon}_s)\,ds+\int_0^t\sigma(s,X^{N,\epsilon}_s)\,dW_s|^2\right].
%
\end{align*}
Defining the positive finite constant $K$ such that
\[
|b(t,x)|^2+|\sigma\sigma^*(t,x)|\leq K(1+|x|^2),\,\forall\,t\in[0,T],
\]
and since
\begin{align*}
\EE_\PP\left[\left|\frac{T^{\mu_{N,\epsilon}(\eta_N(t))\rightarrow \mu^K_{N,\epsilon}(\eta_N(t))}(X^{N,\epsilon}_{\eta_N(t)})-X^{N,\epsilon}_{\eta_N(t)}}{\epsilon}\right|^2\right]=\frac{1}{\epsilon^2}W^2_2(\mu_{N,\epsilon}(\eta_N(t)),K),
\end{align*}
we deduce that
\begin{align*}
&\EE_\PP\left[\max_{t\in[0,T_0]}|X^{N,\epsilon}_t|^2\right]\\
&\leq 8\left(K\vee 1\EE_\PP\left[1+|X^\epsilon_0|^2\right]+\frac{1}{\epsilon^2}\int_0^{T_0}W_2^2(\mu^{N,\epsilon}(\eta_N(t)),K)\,dt +K\int_0^{T_0}\EE_\PP\left[\max_{r\in[0,t]}|X^{N,\epsilon}_r|^2\right]\,dt\right),
\end{align*}
and by Gronwall's lemma, that
\begin{equation}\label{ProofStep1b}
\EE_\PP\left[\max_{t\in[0,T_0]}|X^{N,\epsilon}_t|^2\right]
\leq c\left(\EE_\PP\left[1+|X^\epsilon_0|^2\right]+\int_0^{T_0}W_2^2(\mu^{N,\epsilon}(\eta_N(t)),K)\,dt\right),
\end{equation}
for $c$ a positive finite constant depending only on $T$ and $K$. Using \hypi, we can observe that
\[
W^2_2(\mu_{N,\epsilon}(\eta_N(t)),K)\leq \EE_\PP\left[|X^{N,\epsilon}_{\eta_N(t)}-X^{\epsilon}_0|^2\right],
\]
and that
\begin{align*}
&\int_0^{T_0}\EE_\PP\left[|X^{N,\epsilon}_{\eta_N(t)}-X^{\epsilon}_0|^2\right]\,dt&\\
&=\int_0^{T_0}\EE_\PP\left[\left|\int_0^{\eta_N(t)}\frac{T^{\mu^{N,\epsilon}(\eta_N(s))\rightarrow P_K(\mu^{N,\epsilon}(\eta_N(s)))}(X^{N,\epsilon}_{\eta_N(s)})-X^{N,\epsilon}_{\eta_N(s)}}{\epsilon}\,ds\right.\right.\\
&\quad \left.\left.+\int_0^{\eta_N(t)}b(s,X^{N,\epsilon}_s)\,ds
+\int_0^{\eta_N(t)}\sigma(s,X^{N,\epsilon}_s)\,dW_s\right|^2\right]\,dt&\\
&\leq 8\int_0^{T_0}\left(\int_0^{\eta_N(t)}\frac{1}{\epsilon^2}W^2_2(\mu^{N,\epsilon}(\eta_N(s)),K)\,ds+\int_0^{\eta_N(t)}\EE_\PP\left[|b(s,X^{N,\epsilon}_s)|^2\right]\,ds
+\int_0^{\eta_N(t)}\EE_\PP\left[\sigma\sigma^*(s,X^{N,\epsilon}_s)\right]\,ds\right)\,dt&\\
&\leq 8\int_0^{T_0}\left(\int_0^{t}\frac{1}{\epsilon^2}\EE_\PP\left[|X^{N,\epsilon}_{\eta_N(s)}-X^\epsilon_0|^2\right]\,ds+\int_0^{t}\EE_\PP\left[|b(s,X^{N,\epsilon}_s)|^2\right]\,ds
+\int_0^{t}\EE_\PP\left[\sigma\sigma^*(s,X^{N,\epsilon}_s)\right]\,ds\right)dt.&
\end{align*}
Since $(A_1)$ ensures that
\begin{align*}
&\EE_\PP\left[|b(s,X^{N,\epsilon}_s)|^2\right]+\EE_\PP\left[\sigma\sigma^*(s,X^{N,\epsilon}_s)\right]\\
&\leq 2\EE_\PP\left[|b(s,X^{N,\epsilon}_s)-b(s,X^{\epsilon}_0)|^2\right]+\EE_\PP\left[\sigma\sigma^*(s,X^{N,\epsilon}_s)-\sigma\sigma^*(s,X^{\epsilon}_0)\right]
 +2\EE_\PP\left[|b(s,X^{N,\epsilon}_0)|^2\right]+\EE_\PP\left[\sigma\sigma^*(s,X^{N,\epsilon}_0)\right]\\
&\leq 2 \max_{t\in[0,T]}\left(\Vert b(t,.)\Vert^2_{Lip}+\Vert \sigma\sigma^*(t,.)\Vert_{Lip}\right)\EE_\PP\left[|X^{N,\epsilon}_s-X^{\epsilon}_0|^2\right] +2 Kt\EE_\PP\left[1+|X^{\epsilon}_0|^2\right]
\end{align*}
where $\Vert b(t,.)\Vert_{Lip}$ and $\Vert \sigma\sigma^*(t,.)\Vert_{Lip}$ are the Lipschitz norm of $x\mapsto b(t,x)$ and $x\mapsto \sigma\sigma^*(t,x)$ respectively.
It follows that
\[
\int_0^{T_0}\EE_\PP\left[|X^{N,\epsilon}_{\eta_N(t)}-X^{\epsilon}_0|^2\right]\,dt\leq c\int_0^{T_0} \int_0^t\EE_\PP\left[|X^{N,\epsilon}_{\eta_N(s)}-X^{\epsilon}_0|^2\right]\,ds+KT^2_0\EE_\PP\left[1+|X^{\epsilon}_0|^2\right].
\]
where $c$ is a constant depending only on $\epsilon$,$\Vert b(t,.)\Vert_{Lip}$, $\Vert \sigma\sigma^*(t,.)\Vert_{Lip}$ and $T$.
By Gronwall's lemma, we deduce that
\begin{equation}\label{ProofStep1c}
\int_0^TW^2_2(\mu^{N,\epsilon}(\eta_N(t)),K)\,dt\leq \int_0^{T}\EE_\PP\left[|X^{N,\epsilon}_{\eta_N(t)}-X^{\epsilon}_0|^2\right]\,dt\leq
(1+cTe^{cT})KT^2\EE_\PP\left[1+|X^{\epsilon}_0|^2\right],
\end{equation}
and, coming back to \eqref{ProofStep1b}, that
\begin{equation}\label{ProofStep1d}
\EE_\PP\left[\max_{t\in[0,T]}|X^{N,\epsilon}_t|^2\right]<\infty,\,\forall\,N.
\end{equation}
Additionally,  \eqref{ProofStep1c} ensures that, for all $t\in[0,T]$, $\delta>0$ such that $0\leq t+\delta \leq T$,
\begin{equation}
\label{ProofStep1e}
\begin{aligned}
&\EE_{\PP}\left[|X^{N,\epsilon}_{t+\delta}-X^{N,\epsilon}_{t}|^2\right]\\
&\leq 8\left(\int_{t}^{t+\delta}\EE_{\PP}\left[\left|\frac{T^{\mu^{\epsilon}(\eta_N(s))\rightarrow \mu^K_\epsilon(t_n)}(X^{N,\epsilon}_{t_n})-X^{N,\epsilon}_{\eta_N(s)}}{\epsilon}\right|^2+\left|b(s,X^{N,\epsilon}_s)\right|^2+
\mbox{Trace}\left(\sigma\sigma^*\right)(s,X^{N,\epsilon}_s)\right]\,ds\right)\\
&\leq c\delta.
\end{aligned}
\end{equation}
Let us now consider the process $(Y^{N,\epsilon}_t;\,0\leq t\leq T)$ given by
\[
Y^{N,\epsilon}_t=\argmin_{\mbox{Law}(Y)\in K}\EE_\PP\left[\left|X^{N,\epsilon}_t-Y\right|^2\right].
\]
Observing that
\begin{align*}
\EE_\PP\left[\left|Y^{N,\epsilon}_{t+h}-Y^{N,\epsilon}_t\right|^2\right]&\leq 4\EE_{\PP}\left[\left|Y^{N,\epsilon}_{t+h}-X^{N,\epsilon}_t\right|^2\right] +4\EE_{\PP}\left[\left|Y^{N,\epsilon}_{t+h}-X^{N,\epsilon}_{t+h}\right|^2\right]+4\EE_{\PP}\left[\left|X^{N,\epsilon}_{t+h}-X^{N,\epsilon}_{t}\right|^2\right]\\
&\leq 4W^2_2(\mu_{N,\epsilon}(t),K)+4W^2_2(\mu_{N,\epsilon}(t+h),K)+4h,
\end{align*}
using \eqref{ProofStep1e}. Next, since
\begin{align*}
W_2(\mu_{N,\epsilon}(t+h),K)&\leq  W_2(\mu_{N,\epsilon}(t+h),\mu^K_{N,\epsilon}(t))\\
&\leq W_2(\mu_{N,\epsilon}(t+h),\mu^K_{N,\epsilon}(t))-W_2(\mu_{N,\epsilon}(t),\mu^K_{N,\epsilon}(t))+W_2(\mu_{N,\epsilon}(t),\mu^K_{N,\epsilon}(t))
\end{align*}
we have
\begin{align*}
W_2(\mu_{N,\epsilon}(t+h),K)-W_2(\mu_{N,\epsilon}(t),K)&\leq W_2(\mu_{N,\epsilon}(t+h),\mu^K_{N,\epsilon}(t))-W_2(\mu_{N,\epsilon}(t),\mu^K_{N,\epsilon}(t))\\
&\leq W_2(\mu_{N,\epsilon}(t+h),\mu_{N,\epsilon}(t))\leq c\sqrt{h},
\end{align*}
using the triangular inequality of $W_2$ and \eqref{ProofStep1e} for the last inequality. In the same way, we get
\begin{equation}\label{ProofStep1e''}
\begin{aligned}
&W_2(\mu_{N,\epsilon}(t),K)-W_2(\mu_{N,\epsilon}(t+h),K)\\
&\leq W_2(\mu_{N,\epsilon}(t),\mu^K_{N,\epsilon}(t+h))-W_2(\mu_{N,\epsilon}(t+h),\mu^K_{N,\epsilon}(t+h))\leq c\sqrt{h},
\end{aligned}
\end{equation}
from which we deduce that $t\mapsto W_2(\mu_{N,\epsilon}(t),K)$ is $1/2$-H\"older continuous so that
 \begin{equation}
\EE_\PP\left[\left|Y^{N,\epsilon}_{t+h}-Y^{N,\epsilon}_t\right|^2\right]\leq C h.\label{ProofStep1e'}
\end{equation}
Applying [Karatzas and Shreve \cite{KarShr-98}, Theorem 2.8, Chapter $2$],
we can construct a locally H\"older-continuous modification of $(Y^{N,\epsilon}_t;\,0\leq t\leq T)$, that we still denote $(Y^{N,\epsilon}_t;\,0\leq t\leq T)$ for simplicity.
Combining \eqref{ProofStep1d}, \eqref{ProofStep1e} and \eqref{ProofStep1e'} and since, by triangular inequality, \eqref{ProofStep1d} ensures that
\[
\sup_{N\in\NN}\max_{t\in[0,T]}\EE_\PP\left[|Y^{N,\epsilon}_t|^2\right]<\infty,
\]
the sequence
$\PP^{\epsilon,N}:=\PP\circ(X^{N,\epsilon}_t,W_t,L^{N,\epsilon}_t,Y^{N,\epsilon}_t;\,t\in[0,T])^{-1}$
is relatively compact on $\Pp(\Cc([0,T];\er^{4d}))$ (see e.g. \cite{KarShr-98}, pages 63-64).  Applying Skorohod's representation theorem, there exists a filtered probability space $\left(\tilde{\Omega},\tilde{\Ff},\tilde{\PP}\right)$ under which are defined $(X^{N_k,\epsilon}_t,W^{N_k,\epsilon}_t,L^{N_k,\epsilon}_t,Y^{N,\epsilon}_t;\,0\leq t\leq T)$ such that $\tilde{\PP}\circ \left((X^{N_k,\epsilon}_t,W^{N_k,\epsilon}_t,L^{N_k,\epsilon}_t,Y^{N,\epsilon}_t;\,0\leq t\leq T\right)^{-1}=\PP^{\epsilon,N_k}$, $(X^{\epsilon}_t,L^{\epsilon}_t,W^{\epsilon}_t;\,0\leq t\leq T)$ such that $\tilde{\PP}\circ \left((X^{\epsilon}_t,W^{\epsilon}_t,L^{\epsilon}_t,Y^{\epsilon}_t\right)^{-1}=\PP^{\epsilon,\infty}$,
for which $\tilde{\PP}$-a.s.
\[
\lim_{k\rightarrow \infty}\left((X^{N_k,\epsilon},L^{N_k,\epsilon},W^{N_k,\epsilon},Y^{N,\epsilon}\right)=\left(X^{\epsilon},L^{\epsilon},W^{\epsilon},Y^{\epsilon}\right),\,\mbox{on}\,[0,T].
\]
By continuity of $b$ and $\sigma$,
\[
X^\epsilon_t=X^\epsilon_0+\int_0^t b(s,X^\epsilon_s)\,ds+\int_0^t b(s,X^\epsilon_s)\,dW^\epsilon_s+L^\epsilon_t,
\]
where $X^\epsilon_0\sim \mu_0$, $(W^\epsilon_t;\,0\leq t\leq T)$ is a $\er^d$-Brownian motion and since
\begin{equation}
\label{ProofStep1f}
\liminf_{k\rightarrow\infty}\EE_{\tilde{\PP}}\left[\max_{t\in[0,T]}\left|L^{N_k,\epsilon}_t-\int_0^tY^{N_k,\epsilon}_s\,ds\right|\right]=0,
\end{equation}
we have
\[
L^\epsilon_t=\int_0^t\frac{Y^{\epsilon}_s-X^\epsilon_s}{\epsilon}\,ds
\]
Thus it remains to check that $\tilde{\PP}$-a.s., for all $t\in [0,T]$,
  \begin{equation}
  L^\epsilon_t=\int_0^t \frac{T^{\mu_\epsilon(s)\rightarrow\mu^K_\epsilon(s)}-X^\epsilon_s}{\epsilon}\,ds.\label{ProofStep1e'''}
  \end{equation}
Replicating the proof of \eqref{ProofStep1e''}, we observe that
\[
\left|W_2(\mu_{N_k,\epsilon}(t),K)-W_2(\mu_{\epsilon}(t),K)\right|\leq W_2(\mu_{N_k,\epsilon}(t),\mu_{\epsilon}(t))\leq \EE_{\tilde{\PP}}\left[\left|X^{N_k,\epsilon}_t-X^{\epsilon}_t\right|^2\right]
\]
by Lebesgue's dominated convergence theorem. It then follows that
\[
W_2(\mu_{\epsilon}(t),K)=\EE_{\tilde{\PP}}\left[\left|Y^{\epsilon}_t-X^{\epsilon}_t\right|^2\right]
\]
and by uniqueness of the $W_2$-minimal projection of $\mu^{\epsilon}(t)$ we deduce \eqref{ProofStep1e'''}.
To complete the existence  part, it remains to check that, for all $t$, $\mu^{\epsilon}(t)$ admits a Lebesgue density so that $Y^{\epsilon}_t=T^{\mu_{\epsilon}(t)\rightarrow \mu^K_{\epsilon}(t)}$. To this aim, consider $\RR$ the probability measure of
\[
R_t=X_0+\int_0^tb(s,R_s)\,ds+\int_0^t\sigma(s,R_s)\,dW_s,\,0\leq t\leq T,
\]
and let $\PP_{X^{N,\epsilon}}$ be the law of $(X^{N,\epsilon}_t;\,0\leq t\leq T)$ and $\PP_{X^{\epsilon}}$ be the law of $(X^{\epsilon}_t;\,0\leq t\leq T)$.
Then, by the lower semi-continuity of the relative entropy and \eqref{ProofStep1a},
\begin{align*}
0\leq \EE_{\PP_{X^{\epsilon}}}\left[\log(d\PP_{X^{\epsilon}}/d\RR)\right]&\leq \liminf_{k\rightarrow \infty}\EE_{\PP_{X^{N_k,\epsilon}}}\left[\log(d\PP_{X^{N_k,\epsilon}}/d\RR)\right]\\
&\leq \liminf_{k\rightarrow \infty}\EE_{\PP}\left[\int_0^T\left|\frac{T^{\mu^{\epsilon}(t_n)\rightarrow \mu^K_\epsilon(t_n)}(X^{N,\epsilon}_{t_n})-X^{N,\epsilon}_{t_n}}{\epsilon}\right|^2\,ds\right]
<\infty,
\end{align*}
by \eqref{ProofStep1a}. Therefore $\PP_{X^{\epsilon}}$ is absolutely continuous with  respect to $\RR$ and, in particular,
Hence, for all $t$, $\mu^{\epsilon}(t)$ is absolutely continuous with respect to the Lebesgue measure.
\subsection{Uniqueness part}

Thanks to these preliminaries, we are now in position to prove the strong uniqueness of a solution to \eqref{eq:PenalizedSDE} with the following proposition:

\begin{prop}\label{prop:Uniq} Given $\epsilon>0$, let $(X^1_t, L^1_t;\;0\leq t\leq T)$ and $(X^2_t, L^2_t;\;0\leq t\leq T)$ be two solutions of \eqref{eq:PenalizedSDE}, defined on the same probability space $(\Omega,\Ff,\PP)$, both starting at the same position $X_0$, both driven by the same brownian motion $(W_t;\;0\leq t\leq T)$ and such that their time-marginal distributions are in $\Pp^{ac}_2(\er^d)$ at all time $t$ in $[0,T]$. Then, $\PP$-almost surely, $(X^1_t,L^1_t)=(X^2_t,L^2_t)$, for all $t$ in $[0,T]$.
\end{prop}
\begin{proof}[Proof of Proposition \ref{prop:Uniq}] Owing to assumption \hypcoef, by It\^o's formula, we have, for all $t$
\begin{align*}
\EE\left[\left|X^1_t-X^2_t\right|^2\right]&= 2\EE\left[\int_0^t\left(X^s_t-X^2_s\right)\cdot\left(dL^1_s -dL^2_s\right)\right]\\
&\quad +2\int_0^t\EE\left[\left(b(s,X^1_s)-b(s,X^2_s)\right)\cdot \left(X^s_t-X^2_s\right)\right]\,ds+\sum_{i=1}^d\int_0^t\EE\left[a^{i,i}(s,X^1_s)-a^{i,i}(s,X^2_s)\right]\,ds\\
&\leq 2\EE\left[\int_0^t\left(X^s_t-X^2_s\right)\cdot\left(dL^1_s -dL^2_s\right)\right]+C\int_0^t\EE\left[\left|X^1_s-X^2_s\right|^2\right]\,ds,
\end{align*}
where $C=\Vert b\Vert_{Lip}+\Vert a\Vert_{Lip}$. Since
\begin{align*}
&\EE\left[\int_0^t\left(X^1_t-X^2_s\right)\cdot\left(dL^1_s -dL^2_s\right)\right]\\
&=\frac{1}{\epsilon}\int_0^t\EE\left[\left(X^1_t-X^2_s\right)\cdot\left(\left(T^{\mu_1(s)\rightarrow\mu^K_1(s)}_s-X^1_s\right) -\left(T^{\mu_2(s)\rightarrow\mu^K_2(s)}_s-X^2_s\right)\right)\right]\,ds
\end{align*}
which is nonnegative according to Corollary \ref{coro:Monotone1}. Therefore,
\[
\EE\left[\left|X^1_t-X^2_t\right|^2\right]\leq C\int_0^t\EE\left[\left|X^1_s-X^2_s\right|^2\right]\,ds,
\]
and, by Gronwall's lemma, it follows that $\sup_{t\in[0,T]}\EE\left[\left|X^1_t-X^2_t\right|^2\right]=0$. This immediately implies that, for all $t$ in $[0,T]$,
$\mu_1(t)=\mu_2(t)$, so that $\mu^K_1(t)=\mu^K_2(t)$ and $T^{\mu_1(t)\rightarrow\mu^K_1(t)}_t=T^{\mu_2(t)\rightarrow\mu^K_2(t)}_t$. Therefore, $\PP$-a.s.,
\[
X^1_t-X^2_t=\int_0^t\left(b(s,X^1_s)-b(s,X^2_s)\right)\,ds+\int_0^t\left(\sigma(s,X^1_s)-\sigma(s,X^2_s)\right)\,dW_s,\,\forall\,t\in[0,T],
\]
which, owing to the smoothness of $b$ and $\sigma$, implies that $\max_{t\in[0,T]}\left|X^1_t-X^2_t\right|=0$, $\PP$-a.s. .
\end{proof}
\section{Asymptotic behavior of the penalized system \eqref{eq:PenalizedSDE}}\label{sec:Limit}
In this section, we investigate the limit of the solution to \eqref{eq:PenalizedSDE} as $\epsilon$ tends to $0$. The next subsection is dedicated to the behavior of the time-marginal distributions $(\mu^K_\epsilon(t);\,0\leq t\leq T)$ through the estimate of $W_2(\mu_\epsilon(t),K)$ (see Lemma \ref{lem:ConvMargin} below).
Next, we consider some tightness property related to $(X^\epsilon_t,\,L^\epsilon_t;\,0\leq t\leq T)$
 \subsection{Convergence of the marginal distribution towards the constraint space}
\begin{lemma}\label{lem:ConvMargin} There exists $0<C<\infty$ depending only on $T$, $\max_{1\leq i\leq d}\Vert b^i\Vert_{\infty}$, $\max_{1\leq i\leq d}\Vert a^{i,i}\Vert_\infty$ such that
\[
W^2_2(\mu_\epsilon(t),K)+\frac{1}{\epsilon}\int_0^tW^2_2(\mu_\epsilon(s),K)\,ds\leq C.
\]
\end{lemma}
The preceding lemma ensures that as $\epsilon$ decreases to $0$, the constraint \eqref{eq:WeakConstraint-marginal} in the sense that
\[
\mbox{for a.e. }t\in(0,T), \lim_{\epsilon\rightarrow 0}W_2(\mu_\epsilon(t),K)=0.
\]
\begin{proof} Observing that
\[
W^2_2(\mu_\epsilon(t+h),K)\leq \EE\left[\left|X^\epsilon_{t+h}-T^{\mu_\epsilon(t)\rightarrow \mu^K_\epsilon(t)}(X^\epsilon_t)\right|^2\right],
\]
 then the It\^o's formula applied to $s\mapsto \left|X^\epsilon_{t+h}-T^{\mu_\epsilon(t)\rightarrow \mu^K_\epsilon(t)}(X^\epsilon_t)\right|^2,\,s\in [t,t+h]$, yields
\begin{align*}
&\EE\left[\left|X^\epsilon_{t+h}-T^{\mu_\epsilon(t)\rightarrow \mu^K_\epsilon(t)}(X^\epsilon_t)\right|^2\right]\\
&\leq \EE\left[\left|X^\epsilon_{t}-T^{\mu_\epsilon(t)\rightarrow \mu^K_\epsilon(t)}(X^\epsilon_t)\right|^2\right]+2\EE\left[\int_{t}^{t+h}\left(X^\epsilon_s-T^{\mu_\epsilon(t)\rightarrow \mu^K_\epsilon(t)}(X^\epsilon_t)\right)\cdot\left(\frac{T^{\mu_\epsilon(s)\rightarrow \mu^K_\epsilon(s)}(X^\epsilon_s)-X^\epsilon_s}{\epsilon}\right)\,ds\right]\\
&\quad +2\EE\left[\int_t^{t+h}\left(X^\epsilon_s-T^{\mu_\epsilon(t)\rightarrow \mu^K_\epsilon(t)}(X^\epsilon_t)\right)\cdot b(s,X^\epsilon_s)\,ds\right]+\sum_{i=1}^d\EE \left[\int_t^{t+h} a^{i,i}(s,X^\epsilon_s)\,ds\right].
\end{align*}
Then, according to \hypcoef,
\begin{align*}
&W^2_2(\mu_\epsilon(t+h),K)\\
&\leq W^2_{2}(\mu_\epsilon(t),K)+ dh\max_{1\leq i\leq d}\Vert a^{i,i}\Vert_\infty
+2\EE\left[\int_t^{t+h}\left(X^\epsilon_s-T^{\mu_\epsilon(t)\rightarrow \mu^K_\epsilon(t)}(X^\epsilon_t)\right)\cdot b(s,X^\epsilon_s)\,ds\right]\\
& +2\EE\left[\int_{t}^{t+h}\left(X^\epsilon_s-T^{\mu_\epsilon(t)\rightarrow \mu^K_\epsilon(t)}(X^\epsilon_t)\right)\cdot\left(\frac{T^{\mu_\epsilon(s)\rightarrow \mu^K_\epsilon(s)}(X^\epsilon_s)-X^\epsilon_s}{\epsilon}\right)\,ds\right].
\end{align*}
Dividing the preceding inequality by $h$ and taking the limit $h\rightarrow 0^+$, and by further using
the continuity of $t\mapsto X^\epsilon_t$ and $t\mapsto T^{\mu_\epsilon(t)\rightarrow \mu^K_\epsilon(t)}(X^\epsilon_t)$, we get
\begin{align*}
\frac{d^+}{dt}W^2_2(\mu_\epsilon(t),K) -\frac{W^2_2(\mu_\epsilon(t),K)}{\epsilon}&\leq d\max_{1\leq i\leq d}\Vert a^{i,i}\Vert_\infty
+2\max_{1\leq i\leq d}\Vert b^{i}\Vert_\infty W_2(\mu_\epsilon(t),K)\\
&\leq d\max_{1\leq i\leq d}\Vert a^{i,i}\Vert_\infty
+2\max_{1\leq i\leq d}\Vert b^{i}\Vert_\infty (1+W^2_2(\mu_\epsilon(t),K)).
\end{align*}
Using Gronwall's lemma it follows that
\begin{align*}
\frac{d^+}{dt}W^2_2(\mu_\epsilon(t),K)+\frac{W^2_2(\mu_\epsilon(t),K)}{\epsilon}\leq C(T,\max_{1\leq i\leq d}\Vert b\Vert_{\infty},\max_{1\leq i\leq d}\Vert a^{i,i}\Vert_{\infty}).
\end{align*}
for $\frac{d^+}{dt}$ the right-hand side derivative in time. Integrating the preceding expression on $[0,t]$ gives the result.
\end{proof}
\subsection{Asymptotic behavior of \eqref{eq:PenalizedSDE}}
\begin{theorem}\label{thm:LimitSystem} Under assumptions \hypi to \hypiii, and given $(X^\epsilon_t,L^\epsilon_t;\,t\in[0,T])$, the solution of \eqref{eq:PenalizedSDE} constructed in Theorem \ref{thm:Main1}, there exists a subsequence $(X^{\epsilon_k}_t,L^{\epsilon_k}_t;\,t\in[0,T])$ which, as $\epsilon_k$ tends to $0$, converges in distribution to a couple of continuous stochastic process $(X^0_t,L^0_t;\,t\in[0,T])$ such that
\begin{equation}\label{eq:LimitSystem}
X^0_t=X^0_0+\int_0^tb(s,X^0_s)\,ds+\int_0^t\sigma(s,X^0_s)\,dB_s+L^0_t,\,X_0\sim \mu_0\\
\end{equation}
where $(B_t;\,t\in[0,T])$ is a $\er^d$-Brownian motion and  where $(L^0_t;\,t\in[0,T])$ is a continuous process with bounded variations such that
\begin{align*}
\EE_\PP\left[\int_0^T\left(Y_{s}-X^0_s\right)\cdot dL^0_r\right]\leq 0,
\end{align*}
for all continuous process $(Y_t;\,0\leq t\leq T)$ such that $\mbox{Law}(Y_t)\in K$, for all $0\leq t\leq T$.
\end{theorem}

The proof of Theorem \ref{thm:LimitSystem} will be decomposed into three mains steps: the first step provides a uniform control of the penalization component, derived from \hypiii. The second step concern the tightness of $\PP\circ\left(X^\epsilon_t,L^\epsilon_t;\,0\leq t\leq T\right)$ using the Meyer-Zheng topology and in the last step, a limit point of the sequence has a solution to \eqref{eq:LimitSystem}.

\textbf{Step $1$:}

$\bullet$ \textbf{Uniform control of second moments:} Applying Itô's formula, we have
\begin{equation*}
\EE_\PP\left[|X^\epsilon_t-X_0|^2\right]=\int_0^t\EE_\PP\left[2 b(s,X^\epsilon_s)\cdot\left(X^\epsilon_s-X_0\right)+\sum_{i=1}^d a^{i,i}(s,X^\epsilon_s)\right]\,ds+2\int_0^t\EE_\PP\left[\left(X^\epsilon_s-X_0\right)\cdot dL^\epsilon_s\right].
\end{equation*}
Since $\mu_0\in K$, Proposition \ref{prop:Monotone2} ensures that, for all $0\leq t\leq T$,
\begin{equation*}
0\leq \EE_\PP\left[\left(X^\epsilon_t-X_0\right)\cdot \left(T^{\mu_\epsilon(t)\rightarrow\mu^K_\epsilon(t)}(X^\epsilon_t)-X^\epsilon_t\right)\right] +\EE_\PP\left[\left|X_0-T^{\mu_\epsilon(t)\rightarrow\mu^K_\epsilon(t)}(X^\epsilon_t)\right|^2\right]
\end{equation*}
from which we deduce that
\begin{align*}
0\leq\EE_\PP\left[\int_0^t\left(X^\epsilon_s-X_0\right)\cdot dL^\epsilon_s\right]
+\frac{1}{\epsilon}\int_0^t\EE_\PP\left[\left|X_0-T^{\mu_\epsilon(s)\rightarrow\mu^K_\epsilon(s)}(X^\epsilon_s)\right|^2\right]\,ds.
\end{align*}
Therefore
\begin{align*}
\EE_\PP\left[|X^\epsilon_t-X_0|^2\right]&\leq 2T\left(\max_{1\leq i\leq d}\Vert b^i\Vert_{\infty}+\max_{1\leq i\leq d}\Vert a^{i,i}\Vert_{\infty}\right)+
\int_0^t\EE_\PP\left[2 b(s,X^\epsilon_s)\cdot\left(X^\epsilon_s-X_0\right)+\sum_{i=1}^d a^{i,i}(s,X^\epsilon_s)\right]\,ds\\
&\quad +2\int_0^t\EE_\PP\left[\left|X^\epsilon_s-X_0\right|^2\right]\,ds.
\end{align*}
Applying Gronwall's lemma and since $\EE_\PP\left[|X_0|^2\right]<\infty$, we conclude that
\begin{equation}\label{eq:MomentEstimate}
\sup_{\epsilon>0}\max_{0\leq t\leq T}\EE_\PP\left[|X^\epsilon_t|^2\right]<\infty.
\end{equation}
$\bullet$ \textbf{Uniform control of the penalization component:}
Applying Proposition \ref{prop:Monotone2}, we have
\[
0\geq -\EE_\PP\left[\left(X^\nu-X^\mu\right)\cdot \left(X^{\mu^K}-X^\mu\right)\right]
\]
for any $\nu$ in $K$. Choosing $\nu=\gamma$ for $\gamma$ in $\mbox{Int}(K)$ and a representant $Y$ given by \hypiii, we have
\[
\EE_\PP\left[\left(Y-X^\mu\right)\cdot \left(X^{\mu^K}-X^\mu\right)\right]\geq -r\EE_\PP\left[Z\cdot \left(X^{\mu^K}-X^\mu\right)\right]
\]
for all r.v. $Z$ such that $|Z|\leq 1$. In particular,
\begin{equation}
\label{ProofLemControlStep1}
\EE_\PP\left[\left(Y-X^{\mu_\epsilon(t)}_t\right)\cdot \left(X^{\mu^K_\epsilon(t)}_t-X^{\mu_\epsilon(t)}_t\right)\right]\geq -r
\EE_\PP\left[Z_t\cdot \left(X^{\mu^K_\epsilon(t)}-X^{\mu_\epsilon(t)}_t\right)\right]
\end{equation}
for all adapted process $(Z_t;\,0\leq t\leq T)$ such that $|Z_t|\leq 1$.
From this estimate, we deduce that
\begin{lemma}\label{lem:Control} Assume that \hypiii holds true and let $(X^\epsilon_t,L^\epsilon_t;\,0\leq t\leq T)$ be the solution to \eqref{eq:PenalizedSDE} defined on some probability space $(\Omega,\Ff,\PP)$.  For $Y$ a random variable, also defined on $(\Omega,\Ff,\PP)$, distributed according to $\mathbf{\gamma}$ given in $\mbox{Int}(K)$, we have: for all $t,h>0$ such that $0\leq t\leq t+h\leq T$,
\begin{align*}
&\int_t^{t+h} \frac{1}{\epsilon}\EE\left[\left|T^{\mu_\epsilon(s)\rightarrow \mu^K_\epsilon}(X^\epsilon_s)-X^\epsilon_s\right|^2\right]
\,ds+r\int_t^{t+h}\Vert \frac{T^{\mu^K_\epsilon(s)\rightarrow \mu_\epsilon(s)}-Id}{\epsilon}\Vert_{L^{q}(\mu_\epsilon(s))}\,ds\\
&\leq \left(\EE \left|X^\epsilon_t-Y\right|^2-\EE \left|X^\epsilon_{t+h}-Y\right|^2+
\int_t^{t+h}\EE\left[b(s,X^\epsilon_s)\cdot \left(X^\epsilon_s-Y\right)+\sum_{i=1}^{d}a^{i,i}(s,X^\epsilon_s)\right]\,ds\right),
\end{align*}
where $Id$ is the identity function on $\er^d$: $Id(x)=x$.
\end{lemma}
\begin{proof}
Dividing both sides of the \eqref{ProofLemControlStep1} by $\epsilon$ and integrating over $[t,t+h]$ for $0<t<T$, we obtain that
\begin{align*}
&-\frac{r}{\epsilon}\int_t^{t+h}\EE\left[ Z_s\cdot \left(T^{\mu_\epsilon(s)\rightarrow\mu^K_\epsilon(s)}(X^\epsilon_s)-X^\epsilon_s\right)\right]\,ds\\
&\leq \int_t^{t+h}\EE\left[\left(X^\epsilon_t-Y\right)\cdot \left(X^\epsilon_s-T^{\mu_\epsilon(s)\rightarrow\mu^K_\epsilon(s)}(X^\epsilon_s)\right)\right]\,ds.
\end{align*}
Applying It\^o's formula, the right-hand side of the above expression is equal to
\begin{align*}
&\EE\left[\left|X^\epsilon_t-Y\right|^2\right]-\EE\left[\left|X^\epsilon_{t+h}-Y\right|^2\right]\\
&-\int_t^{t+h}\EE\left[b(s,X^\epsilon_s)\cdot\left(\frac{T^{\mu_\epsilon(s)\rightarrow\mu^K_\epsilon(s)}(X^\epsilon_s)-X^\epsilon_s}{\epsilon}\right) + \sum_{i=1}^{d}a^{i,i}(s,X^\epsilon_s)\right]\,ds.
\end{align*}
For the left-hand side, choosing
\[
Z_t=\mbox{sign}\left(X^\epsilon_s-T^{\mu_\epsilon(s)\rightarrow\mu^K_\epsilon(s)}(X^\epsilon_s)\right)
\]
for $\mbox{sign}(x)$ the sign function, yields the estimate.
\end{proof}
\paragraph{Step $2$: Tightness result:} Following Lemma \ref{lem:Control}, we can deduce the tightness of the law $\{P^\epsilon:=\PP\circ\left(X^\epsilon_t,L^\epsilon_t;\,0\leq t\leq T\right)\}$. To this aim, we follow and slightly adapt the proof argument of \cite{Cepa-98}: Define
$$
\theta^\epsilon(t)=t+\EE_\PP\left[|L^\epsilon|_t\right]=t+\EE_\PP\left[\int_0^t\left|dL^\epsilon_t/dt\right|\,dr\right],\,0\leq t\leq T,
$$
its generalized inverse
$$
\tau^\epsilon(t)=\inf\left\{r>0\,|\,\theta^\epsilon(r)>t\right\},\,0\leq t\leq T,
$$
and the processes
$$
\widehat{X}^\epsilon_t=X^\epsilon_{\tau^\epsilon(t)},\,\widehat{L}^\epsilon_t=L^\epsilon_{\tau^\epsilon(t)},\,\widehat{A}^\epsilon_t=A^\epsilon_{\tau^\epsilon(t)},
\,\widehat{M}^\epsilon_t=M^\epsilon_{\tau^\epsilon(t)},\,0\leq t\leq T,
$$
where
$$
A^\epsilon_t=\int_0^tb(s,X^\epsilon_s)\,ds,\,M^\epsilon_t=\int_0^t\sigma(s,X^\epsilon_s)\,dW_s
$$
Owing to Lemma \ref{lem:Control} and  since $t\rightarrow T^{\mu_\epsilon(t)\rightarrow \mu^K_\epsilon(t)}(X^\epsilon_t)$ is continuous,
$\{\theta^\epsilon(t);\,0\leq t\leq T\}$ is a sequence of increasing continuous such that
\[
\sup_{\epsilon>0}\max_{0\leq t\leq T}|\theta_\epsilon(t)|\leq T+\sup_{\epsilon>0}\EE_\PP\left[|L^\epsilon|_T\right]<\infty
\]
and thus is relatively compact on the space of non-negative measures defined on $[0,T]$.
 since $t\rightarrow T^{\mu_\epsilon(t)\rightarrow \mu^K_\epsilon(t)}(X^\epsilon_t)$ is continuous, we have, for all $0\leq s\leq t\leq T$,
 $$
 \left|\tau_\epsilon(t)-\tau_\epsilon(s)\right|=\left|\int_s^t\frac{1}{1+\EE_\PP\left[|dL^\epsilon/dr|_{\tau^\epsilon(r)}\right]}\,dr\right|\leq t-s
 $$
so that $\{\theta^\epsilon(t);\,0\leq t\leq T\}_{\epsilon>0}$ is a family of equi-continuous functions on $\Cc([0,T];[0,\infty))$ and thus is relatively compact in $\Cc([0,T];[0,\infty))$.

We further can observe that, for all $\beta>0$, for all $\epsilon>0$,
\begin{align*}
\EE_\PP\left[\max_{0\leq s,t\leq T,\,|t-s|\leq \beta}\left|\widehat{A}^\epsilon_t-\widehat{A}^\epsilon_s\right|\right]&=
\EE_\PP\left[\max_{0\leq s,t\leq T,\,|t-s|\leq \beta}\left|\int_{\tau_\epsilon(s)}^{\tau_\epsilon(t)} b(r,X^\epsilon_r)\,dr\right|\right]\\
&\leq \max_{1\leq i\leq d}\Vert b^i\Vert_{\infty}\max_{0\leq s,t\leq T,\,|t-s|\leq \beta}\left|\tau_\epsilon(t)-\tau_\epsilon(t)\right|\leq \sup_{1\leq i\leq d}\Vert b^i\Vert_{\infty}\beta,
\end{align*}
and
\begin{align*}
\EE_\PP\left[\max_{0\leq s,t\leq T,\,|t-s|\leq \beta}\left|\widehat{M}^\epsilon_t-\widehat{M}^\epsilon_s\right|\right]
&=\EE_\PP\left[\max_{0\leq s,t\leq T,\,|t-s|\leq \beta}\left|\int_{\tau_\epsilon(s)}^{\tau_\epsilon(t)}  \sigma(r,X^\epsilon_r)\,dW_r\right|\right] \\
&\leq \sqrt{\max_{1\leq i,j\leq d}\Vert a^{i,j}\Vert}\sqrt{\left|\tau_\epsilon(t)-\tau_\epsilon(t)\right|}\leq \sqrt{\max_{1\leq i,j\leq d}\Vert a^{i,j}\Vert}\beta.
\end{align*}
We also observe that
\begin{align*}
\EE_\PP\left[ \left|\widehat{L}^\epsilon_t-\widehat{L}^\epsilon_s\right|\right]&=
\EE_\PP\left[\left|\int_{\tau^\epsilon(s)}^{\tau^\epsilon(t)}dL^\epsilon_s\right|\right]=
\EE_\PP\left[\left|\int_{s}^{t}\tau^\epsilon(r) dL^\epsilon/dr_{\tau^\epsilon(r)}\,dr\right|\right]\\
&= \EE_\PP\left[\left|\int_s^t \frac{dL^\epsilon/dr_{\tau^\epsilon(r)}}{1+\EE_\PP\left[|dL^\epsilon/dr_{\tau^\epsilon(r)}|\right]}\,dr\right|\right]\leq |t-s|.
\end{align*}
where the second equality follows by a change of variables. Combining the two previous estimates, we get that
\begin{lemma}\label{lem:Tightness}
The family of probability measures $\{P^\epsilon\}_{\epsilon>0}$ generated by $(\widehat{X}^\epsilon_t,\,\widehat{L}^\epsilon_t,\,\widehat{A}^\epsilon_t,\,\widehat{M}^\epsilon_t;\,0\leq t\leq T)$ is tight in $\Cc([0,T];\er^{5d})$.
\end{lemma}
\textbf{Step 3: Identification of the limit system}
Due to the preceding tightness result and using the Skorokhod representation theorem, we can extract a filtered probability space $(\Omega,\Ff,\PP)$ and a converging subsequence $\{\left(\widehat{X}^{\epsilon_k}_t,\,\widehat{L}^{\epsilon_k}_t,\,\widehat{R}^{\epsilon_k}_t,\,W_t;\,0\leq t\leq T\right)\}_{k\in \NN}$ such that $\PP$-a.s., as $\epsilon_k$ tends to $0$,
$$
\left(\widehat{X}^{\epsilon_k}_t,\,\widehat{L}^\epsilon_t,\,\widehat{A}^\epsilon_t,\,\widehat{M}^\epsilon_t;\,0\leq t\leq T\right)
$$
converges, uniformly in $[0,T]$, to
$$
\left(\widehat{X}^0_t,\,\widehat{L}^0_t,\,\widehat{A}^0_t,\,\widehat{M}^0_t\,;\,0\leq t\leq T\right).
$$
Let us further define $(\theta_0(t);\,0\leq t\leq T)$ and $(\tau_0(t);\,0\leq t\leq T)$ the corresponding limit points (up to the extracting of a further sub-sequence) of $(\theta_{\epsilon_k}(t);\,0\leq t\leq T)$ and $(\tau_{\epsilon_k}(t);\,0\leq t\leq T)$. $t\in[0,T]\mapsto\theta_0(t)$ is an increasing Lipschitz continuous function
(with $\sup_{t\neq s}|\theta_0(t)-\theta_0(s)|/|t-s|=1$) and $t\in[0,T]\mapsto\tau_0(t)$ is an increasing continuous function. Additionally $\tau_0(0)=0=\theta_0(0)$ and $\tau_0$ remains the generalized inverse of $\theta_0$. We denote below $d\tau_0$ and $d\theta_0$ the finite measure on $[0,T]$ generated by $\tau_0$ and $\theta_0$ respectively, with satisfy
\[
\lim_{\epsilon\rightarrow 0}\int_0^Tf(s)d\tau_\epsilon(s)=\int_0^Tf(s)d\tau_0(s),\,\,\lim_{\epsilon\rightarrow 0}\int_0^Tf(s)d\theta_\epsilon(s)=\int_0^Tf(s)d\theta_0(s),
\]
for all continuous function $f:[0,T]\rightarrow \er$.

Let us now define the processes $\left(X^0_t,L^0_t,A^0_t,M^0_t;\,0\leq t\leq T\right)$ as
$$
X^0_t=\widehat{X}^0_{\theta^0(t)},\,L^0_t=\widehat{L}^0_{\theta^0(t)},\,A^0_t=\widehat{A}^0_{\theta^0(t)},\,M^0_t=\widehat{M}^0_{\theta^0(t)}\,0\leq t\leq T,
$$
Since $\tau_0$ is increasing, there is no interval $[s,t]\subset [0,T]$ on which $\tau^0$ is constant. According to [Kurtz \cite{Kurtz-91}, Lemma 2.3 (b)] this ensures that, $\PP$-a.s.
\begin{equation}\label{eq:Convergence}
\lim_{\epsilon\rightarrow 0}\left(X^\epsilon_t,L^\epsilon_t,A^\epsilon_t,M^\epsilon_t;\,0\leq t\leq T\right)=\left(X^0_t,L^0_t,A^0_t,M^0_t;\,0\leq t\leq T\right)
\end{equation}
where the converge holds true in the space of c\`adl\`ag functions equipped with the Skorohod topology. As the limit processes have all continuous paths,
the convergence still holds true in $\Cc([0,T];\er^{4d})$. \eqref{eq:Convergence} ensures that
\begin{equation*}
X^0_t=X^0_0+A^0_t+M^0_t+L^0_t,\,0\leq t\leq T.
\end{equation*}
and, since $b$ and $\sigma$ are continuous and bounded, \eqref{eq:Convergence} also ensures that, for all $0\leq t\leq T$
\[
A^0_t=\lim_{\epsilon\rightarrow 0}\int_0^tb(s,X^\epsilon_s)\,ds=\int_0^t b(s,X^0_s)\,ds
\]
and
\[
M^0_t=\lim_{\epsilon\rightarrow 0}\int_0^t \sigma(s,X^\epsilon_s)\,dW_s=\int_0^t \sigma(s,X^0_s)\,dW_s.
\]
According to Lemma \ref{lem:ConvMargin}, $\mbox{Law}(X^\epsilon_t)$ belongs to $K$ for all $0\leq t\leq T$. Hence, it remains to show that
  $(L^0_t;\,0\leq t\leq T)$ is a continuous process with finite variations, such that,
\[
0\geq -\EE_\PP\left[\int_0^T \left(Y_s-X^0_s\right)\cdot dL^0_s\right]
\]
for all continuous process $(Y_t;\,0\leq t\leq T)$ such that, for all $0\leq t\leq T$, $\mbox{Law}(Y_t)$ is in $K$. To this aim, let us observe that for $|\widehat{L}^\epsilon|$
\[
\sup_{\epsilon>0}\EE_\PP\left[|\widehat{L}^\epsilon|_T\right]\leq T
\]
where $|\widehat{L}^\epsilon|$ is the total variations on $[0,T]$. Therefore,
the composition $(\widehat{L}^0_{\theta^0(t)};,t\leq T)$ is a continuous process with finite variations.

In addition,  for all $0\leq s\leq t\leq T$,
\begin{align*}
0\geq -\EE_\PP\left[\int_{\theta_\epsilon(s)}^{\theta_\epsilon(t)}\left(Y_r-X^\epsilon_r\right)\cdot dL^\epsilon_r\right]+\frac{1}{\epsilon}\EE_\PP\left[\int_{\theta_\epsilon(s)}^{\theta_\epsilon(t)}\left|Y_s-T^{\mu_\epsilon(r)\rightarrow \mu^K_\epsilon(r)}(X^\epsilon_r)\right|^2\,dr\right]
\end{align*}
gives, by change of variables, that
\begin{align*}
0\geq -\EE_\PP\left[\int_{s}^{t}\left(Y_{\tau_\epsilon(r)}-\widehat{X}^\epsilon_r\right)\cdot d\widehat{L}^\epsilon_r\right]+\frac{1}{\epsilon}\EE_\PP\left[\int_{s}^{t}\left|Y_{\tau_\epsilon(r)}-T^{\mu_\epsilon(\tau_\epsilon(r))\rightarrow \mu^K_\epsilon(\tau_\epsilon(r))}(\widehat{X}^\epsilon_r)\right|^2\,dr\right].
\end{align*}
Since $(Y_t;\,0\leq t\leq T)$ has continuous paths, taking the limit $\epsilon\rightarrow 0$ yields
\begin{align*}
0\geq -\EE_\PP\left[\int_{s}^{t}\left(Y_{\tau_0(r)}-\widehat{X}^0_r\right)\cdot d\widehat{L}^0_r\right].
\end{align*}
and furthermore
\begin{align*}
0\geq -\EE_\PP\left[\int_{\theta_0(s)}^{\theta_0(t)}\left(Y_{\tau_0(r)}-\widehat{X}^0_r\right)\cdot d\widehat{L}^0_r\right]=-\EE_\PP\left[\int_s^ t\left(Y_{r}-X^0_r\right)\cdot dL^0_r\right].
\end{align*}
which concludes the proof of Theorem \ref{thm:Main2}.

For the last property, observe that, for all $0\leq s\leq t\leq T$,
\begin{align*}
0\leq -\EE_\PP\left[\int_{\theta_\epsilon(s)}^{\theta_\epsilon(t)}\left(Y_r-X^\epsilon_r\right)\cdot dL^\epsilon_r\right]+\frac{1}{\epsilon}\EE_\PP\left[\int_{\theta_\epsilon(s)}^{\theta_\epsilon(t)}\left|Y_s-T^{\mu_\epsilon(r)\rightarrow \mu^K_\epsilon(r)}(X^\epsilon_r)\right|^2\,dr\right]
\end{align*}
gives, by change of variables, that
\begin{align*}
0\leq -\EE_\PP\left[\int_{s}^{t}\left(Y_{\tau_\epsilon(r)}-\widehat{X}^\epsilon_r\right)\cdot d\widehat{L}^\epsilon_r\right]+\frac{1}{\epsilon}\EE_\PP\left[\int_{s}^{t}\left|Y_{\tau_\epsilon(r)}-T^{\mu_\epsilon(\tau_\epsilon(r))\rightarrow \mu^K_\epsilon(\tau_\epsilon(r))}(\widehat{X}^\epsilon_r)\right|^2\,dr\right].
\end{align*}
Since $(Y_t;\,0\leq t\leq T)$ has continuous paths, taking the limit $\epsilon\rightarrow 0$ yields
\begin{align*}
0\leq -\EE_\PP\left[\int_{s}^{t}\left(Y_{\tau_0(r)}-\widehat{X}^0_r\right)\cdot d\widehat{L}^0_r\right].
\end{align*}
and furthermore
\begin{align*}
0\leq -\EE_\PP\left[\int_{\theta_0(s)}^{\theta_0(t)}\left(Y_{\tau_0(r)}-\widehat{X}^0_r\right)\cdot d\widehat{L}^0_r\right]=\EE_\PP\left[\int_s^ t\left(Y_{r}-X^0_r\right)\cdot dL^0_r\right].
\end{align*}
This ends the proof of Theorem \ref{thm:Main2}.
\begin{footnotesize}

\end{footnotesize}
\end{document}